\documentclass[leqno,final]{siamltex}
\usepackage{graphicx} 
\usepackage{amsmath,amstext,amssymb,bm}
\usepackage{leftidx}
\numberwithin{equation}{section}
\usepackage{xcolor}
\usepackage{soul}
\usepackage{tikz}
\usetikzlibrary{shapes,arrows}
\usepackage{mathrsfs}

\newtheorem{remark}{Remark}[section]

\allowdisplaybreaks[4]

\newcommand{\eps}{\varepsilon}
\newcommand{\Ome}{{\Omega}}

\newcommand{\tH}{\widetilde{H}}
\newcommand{\tW}{\widetilde{W}}

\newcommand{\R}{\mathbb{R}}

\newcommand{\N}{\mathbb{N}}

\begin{document}

\title{On New Families of Fractional Sobolev Spaces\thanks{This work was partially supported by the NSF grant DMS-1620168.} }
\markboth{X. FENG and M. SUTTON}{A NEW THEORY OF FRACTIONAL SOBOLEV SPACES}

%
	
\author{Xiaobing Feng\thanks{Department of Mathematics, The University of Tennessee, 
Knoxville, TN 37996. U.S.A. (xfeng@math.utk.edu).}
\and{Mitchell Sutton}\thanks{Department of Mathematics, The University of Tennessee, 
Knoxville, TN 37996. U.S.A. (msutto11@vols.utk.edu).} }

\date{}

\maketitle
 
\thispagestyle{empty}

\begin{abstract}
 This paper presents three new families of fractional Sobolev spaces 
    	and their accompanying theory in one-dimension. The new construction 
    	and theory are based on a newly developed notion of weak fractional derivatives, 
    	which are natural generalizations of the well-established integer order Sobolev 
    	spaces and theory. In particular, two new families of one-sided fractional 
    	Sobolev spaces 
        are introduced and analyzed,  they reveal more insights about another family of   
        so-called symmetric fractional Sobolev spaces. Many key theorems/properties, such as density/approximation theorem, extension theorems, one-sided trace theorem, and 
        various embedding theorems and Sobolev inequalities in those Sobolev spaces are 
        established. Moreover, 
        a few relationships with existing fractional Sobolev spaces are also discovered. 
       The results of this paper lay down a solid theoretical foundation 
       for systematically developing a fractional calculus of 
       variations theory and a fractional PDE theory as well as their numerical 
       solutions in subsequent works. 
       This paper is a concise presentation of the materials of Sections 1, 4 and 5 of reference \cite{Feng_Sutton}. 
\end{abstract}




 

\section{Introduction}\label{sec-1}
Fractional Sobolev spaces have been known for many years 
	(cf. \cite{Adams,Brezis,Lions}, also see \cite{Nezza}), they are the cornerstone and provide 
	an important functional 
	setting for studying boundary value problems of partial differential equations (PDEs)  \cite{Evans,Trudinger,Lions}. 
    In recent years fractional Sobolev spaces, along with     
    fractional calculus and fractional order differential equations, has garnered a lot 
    of interest and attention both from the PDE community and in the applied mathematics 
    and scientific communities. Besides the genuine mathematical interest and curiosity, 
    this trend has also been driven by intriguing scientific and engineering applications which give rise to fractional order differential equation models to better describe the (time) 
     memory effect and the (space) nonlocal phenomena (cf. \cite{Caffarelli07,Du19, Guo,Hilfer,Kilbas, Meerschaert} and the references therein). It is the rise of these applications that revitalizes the field of fractional calculus and fractional differential equations and calls for further research in the field, including to develop new numerical methods for solving various fractional order problems. 
     
     Historically, the existing fractional order Sobolev spaces were primarily introduced as 
     a functional framework to study boundary value problems of integer order PDEs in general 
     bounded domains \cite{Lions} (also see \cite{Adams,Nezza,Evans}). Although they have been 
     successfully used to analyze certain 
     fractional order differential equations (cf. \cite{Caffarelli07,Du19,Ervin,Du19,Guo,Podlubny} and the references therein), some issues and limitations of using them to study more general 
     fractional order differential equations have been raised and exposed (cf. \cite{Malinowska,Samko}), in particular, when domain-dependent fractional 
     order differential operators are involved. 
     Motivated by such a challenge/need, in a previous work \cite{Feng_Sutton1} (also see \cite{Feng_Sutton}), the authors of 
     this paper introduced a new fractional differential calculus theory, in which the notion of \textit{weak fractional derivatives} was introduced, and its calculus rules, such as product 
     and chain rules, and the \textit{Fundamental Theorem of Weak Fractional Calculus} (FTwFC) 
     were established. Moreover, 
     many basic properties, such as linearity, semigroup property, inclusivity, and consistency
     were proved and several characterizations of weakly fractional differentiable functions 
     were explored; including the all-important characterization by smooth functions. The new 
     weak fractional differential calculus theory serves as a unifying concept in light of 
     the muddled classical fractional calculus with its numerous (none equivalent) definitions 
     and loss of basic calculus rules. It is our aim to use the newly introduced weak 
     fractional derivative notion to develop the required function spaces for studying 
     general fractional order differential equations in a systematic way similar to that their
     integer order counterparts have been done.
 
    The primary goal of this paper is to develop some new families of fractional Sobolev 
    spaces and their accompanying theory in one-dimension. Unlike the existing 
    fractional Sobolev space theories,  our construction 
    and theory are based on the newly developed notion of weak fractional derivatives, 
    that are analogous to the integer order Sobolev spaces and theory. In particular, 
    two new families of one-sided domain-dependent fractional 
    Sobolev spaces are introduced and analyzed, they reveal more insights about another family of   
    so-called symmetric fractional Sobolev spaces. As in the integer order case, 
    the focuses of this study are to establish key theorems/properties in those new fractional 
    Sobolev spaces, such as density/approximation theorem, extension theorems, 
    one-sided trace theorem, and various embedding theorems and Sobolev inequalities. 
    It is expected that the results of this paper lay down a solid theoretical foundation 
    for systematically developing a fractional calculus of 
    variations theory and a fractional PDE theory as well as their numerical 
    solutions in subsequent works. 
    
    The remainder of this paper is organized as follows.  In Section \ref{sec-2} we
    introduce some preliminaries including to recall two widely used definitions of 
    existing fractional Sobolev spaces, and the definitions of weak fractional 
    derivatives and their characterizations. In Section \ref{sec-3} we first introduce our 
    new families of fractional Sobolev spaces using weak fractional derivatives
    in exactly the same spirit as the integer order Sobolev spaces were defined. 
    We then collect a few elementary properties of those spaces. 
    Section \ref{sec-4} is devoted to the establishment of a fractional Sobolev space theory 
    that is analogous to the theory found in the integer order case, which consists of proving 
    a density/approximation theorem, extension theorems, a one-sided trace theorem, 
    various embedding theorems and Sobolev inequalities.
    Moreover, a few connections between the new fractional Sobolev spaces and existing fractional Sobolev spaces are also established.
    Finally, the paper is concluded by a short summary and a few concluding remarks 
    given in Section \ref{sec-5}.
 
\section{Preliminaries}\label{sec-2}
Let $\R:=(-\infty, \infty)$.  Throughout this paper $\Omega$ denotes either a finite interval $(a,b)\subsetneq\R$ or the whole real line $\R$.  $\Gamma: \R\to \R$ denotes the 
standard Gamma function and $\N$ stands for the set of all positive integers. In addition, 
$C$ will be used to denote a generic positive constant which may be different at different 
locations and $f^{(n)}$ denotes the $n$th order classical derivative of $f$ for $n\in\N$. 
Unless stated otherwise, all integrals $\int_a^b \varphi(x)\, dx$ are understood as Lebesgue  integrals. $L^p(\Ome)$ for $1\leq p\leq \infty$ denotes the standard $L^p$ space. $(\cdot,\cdot)$
denotes the standard $L^2$-inner product. 
Also throughout this paper we shall use the convention $\hat{u}: = \mathcal{F}[u]$ to denote 
the Fourier transform of a given function $u$ on $\R$. 

Moreover,  ${^{-}}{D}{^{\alpha}}$ and ${^{+}}{D}{^{\alpha}}$ denote respectively any left 
and right $\alpha (>0)$ order classical fractional derivatives equivalent to the Riemann-Liouville fractional derivative on the space $C^{\infty}_{0}(
\R)$; this includes Caputo, 
Fourier, and Gr\"{u}nwald-Letnikov fractional derivatives (cf. \cite{Samko}, also see \cite[Section 2]{Feng_Sutton}). ${^{\pm}}{D}{^{\alpha}}$ denotes either ${^{-}}{D}{^{\alpha}}$ or ${^{+}}{D}{^{\alpha}}$. In the case $\Ome=(a,b)$, for any $\varphi \in C^{\infty}_{0}(\Omega)$, $\tilde{\varphi}$ is used to denote the zero extension of $\varphi$ to $\R$.
 
\subsection{Two Existing Definitions of Fractional Sobolev Spaces}\label{sec-5.1}
 Three major definitions of fractional order Sobolev spaces have been known in the literature.
 Below we will only recall two relevant definitions. For the 
 third definition, we refer the reader to \cite{Adams,Lions} for details. 

\begin{definition}
    Let $\Omega \subseteq \R$,   $s>0$, and   $1 \leq p \leq  \infty$. Set $m:=[s]$ and $\sigma:=s-m$. 
    Define the fractional Sobolev space ${\tW}^{s,p}(\Omega)$ by
    \[
    {\tW}^{s,p}(\Omega) : = \left\{ u \in W^{m,p}(\Omega) ; \dfrac{\left|\mathcal{D}^m u(x) - \mathcal{D}^m u(y) \right|}{|x- y|^{\frac{1}{p} + \sigma}} \in L^{p}(\Omega \times \Omega)\right\}, 
    \]
    which is endowed with the norm 
    \begin{align*}
    \|u\|_{{\tW}^{s,p}(\Omega)} : =\begin{cases} \displaystyle{ 
    \Bigl( \|u\|_{W^{m,p}(\Omega)}^{p} + [\mathcal{D}^m u]_{{\tW}^{\sigma,p}(\Omega) }  \Bigr)^{\frac{1}{p}} } &\qquad\mbox{if } 1\leq p< \infty, \\
    \displaystyle{  \|u\|_{W^{m,\infty}(\Omega)}  + [\mathcal{D}^m u]_{{\tW}^{\sigma,\infty}(\Omega) } }    &\qquad\mbox{if } p= \infty,
     \end{cases}
    \end{align*}
    where
    \begin{equation*}
    [u]_{{\tW}^{\sigma,p}(\Omega) }: =  \begin{cases} \displaystyle{
    \Bigl( \int_{\Omega} \int_{\Omega} \dfrac{|u(x) - u(y)|^{p}}{|x-y|^{1 + \sigma p}} \Bigr)^{\frac{1}{p}}\, dx dy }
    &\qquad\mbox{if } 1\leq p< \infty,\\
    \displaystyle{ \sup_{(x,y)\in \Omega\times \Omega} \dfrac{|u(x) - u(y)|}{|x-y|^{\sigma } } }
    &\qquad\mbox{if } p=\infty.  
    \end{cases} 
    \end{equation*}
    When $p=2$, we set ${\tH}^{s}(\Omega) := {\tW}^{s,2}(\Omega)$. 
\end{definition}

When $\Ome=\R$, the following definition based on the Fourier transform is popular.

\begin{definition}
    Let   $s>0$ and  $1 \leq p \leq  \infty$. 
     Define the fractional Sobolev space  $\widehat{W}^{s,p}(\R)$ by
    \[
    \widehat{W}^{s,p}(\R) : = \left\{ u \in L^{p}(\R) :  [u]_{\widehat{W}^{s,p}(\R)} < \infty \right\},  \qquad 1\leq p\leq  \infty,
    \]
   where 
   \[
   [u]_{\widehat{W}^{s,p}(\R)} : = \int_{\R} (1+ |\xi|^{s p}) |\hat{u} (\xi)|^{p}\,d\xi ,
   \qquad 1\leq p\leq \infty.
   \]
   When $p=2$, we set $\widehat{H}^{s}(\R) := \widehat{W}^{s,2}(\R)$.
\end{definition}

\begin{remark}
(a) It is well known (cf. \cite{Adams}, \cite{Nezza}) that  ${\tW}^{s,p}(\Omega)$ and $\widehat{W}^{s,p}(\R)$ are Banach spaces,
	and ${\tH}^{s}(\Omega)$ and  $\widehat{H}^{s}(\R)$ are Hilbert spaces.  
	
(b) It is also well known (cf. \cite{Adams},\cite{Nezza}) that ${{\tH}^{s}(\R)}$ and $\widehat{H}^{s}(\R)$ are equivalent spaces. In particular,
\begin{align}\label{SeminormRelation}
    [u]_{{\tH}^{s}(\R)} \cong \int_{\R} |\xi|^{2 s} |\hat{ u}(\xi)|^{2}\,d\xi.
\end{align}
However, ${{\tW}^{s,p}(\R)}$ and $\widehat{W}^{s,p}(\R)$ are not equivalent spaces for $p \neq 2$. 

(c) Although the definitions above have some kind of differentiability built in, neither of 
them are analogous to the definitions used in the integer order case which are constructed 
using weak derivatives.
\end{remark}

\subsection{Weak Fractional Derivatives}\label{sec-2.2}
Like in the integer order case, the idea of \cite{Feng_Sutton,Feng_Sutton1} to define {\em weak} fractional derivatives ${^{\pm}}{ \mathcal{D}}{^{\alpha}} u$ of a function $u$ is to specify its action on any smooth compactly supported function $\varphi \in C^{\infty}_{0}(\Omega)$, 
instead of knowing its pointwise values as done in the classical fractional 
derivative definitions.

\begin{definition}\label{RWFD}
	For $\alpha> 0$, let $[\alpha]$ denote the integer part of $\alpha$. For $u \in L^{1}(\Omega)$, 
	\begin{itemize} 
		\item[{\rm (i)}] a function $v \in L_{loc}^{1} (\Omega)$ is called the left weak fractional derivative of $u$ if 
		\begin{align*}
		\int_{\Omega} v(x) \varphi(x) \,dx = (-1)^{[\alpha]} \int_{\Omega} u(x) {^{+}}{D}{^{\alpha}} \tilde{\varphi}(x) \, dx
		\qquad \forall \varphi \in C_{0}^{\infty} (\Omega),
		\end{align*}
		we write ${^{-}}{ \mathcal{D}}{^{\alpha}} u:=v$; 
		\item[{\rm (ii)}] a function $w\in L_{loc}^{1} (\Omega)$ is called the right weak fractional derivative of $u$ if 
		\begin{align*}
		\int_{\Omega} w(x) \varphi(x) \,dx = (-1)^{[\alpha]} \int_{\Omega} u(x) {^{-}}{D}{}^{\alpha} \tilde{\varphi}(x) \,dx
		\qquad \forall \varphi \in C_{0}^{\infty} (\Omega), 
		\end{align*}
		and we write ${^{+}}{\mathcal{D}}{^{\alpha}} u:=w$. 
	\end{itemize}
\end{definition}

\begin{remark}
It is easy to check \cite{Feng_Sutton,Feng_Sutton1} that the above weak fractional
derivatives are well defined. It should also be noted that the above definition appears  
to be exactly the same as that of the integer order case, however, there is a foundational
difference, that is, ${^{\pm}}{D}{^{\alpha}} \tilde{\varphi}$ are not 
compactly supported anymore because of the pollution (or nonlocal) effect of fractional 
order derivatives, which causes all the major difficulties in the weak 
fractional differential calculus \cite{Feng_Sutton1} and in this paper. 
\end{remark}

We conclude this section by quoting the following characterization theorem of weak 
fractional derivatives and the \textit{Fundamental Theorem of Weak Fractional Calculus} (FTwFC). Proofs can be found in  \cite[Theorem 4.1 and 4.2]{Feng_Sutton} and \cite[Theorem 4.5]{Feng_Sutton}, respectively.

\begin{theorem}\label{characterization}
	Let $\Ome=(a,b)$ or $\R$ and $u \in L^{1}(\Omega)$. Then $v = {^{\pm}}{\mathcal{D}}{^{\alpha}} u \in L^{1}_{loc} (\Omega)$ if and only if there exists a sequence $\left\{u_j \right\}_{j=1}^{\infty} \subset C^{\infty} (\Omega)$ such that $u_j \rightarrow u$ in $L^{1}(\Omega)$ and ${^{\pm}}{\mathcal{D}}{^{\alpha}} u_j \rightarrow v$ in $L^{1}_{loc}(\Omega)$ as $j \rightarrow \infty$. 
\end{theorem}

\begin{theorem}\label{FTWFC}
       Let $\Omega=(a,b)\subset \R$ and $0 < \alpha <1$. Suppose that $u \in L^{p}(\Omega)$ and  ${^{\pm}}{\mathcal{D}}{^{\alpha}}u \in L^{p}(\Omega)$ for some $1\leq p < \infty$. Then
       there holds
       \begin{align}\label{WeakFTFC}
           u = c^{1-\alpha}_{\pm} \kappa^{\alpha}_{\pm}  + {^{\pm}}{I}{^{\alpha}}{^{\pm}}{\mathcal{D}}{^{\alpha}} u \qquad
           \mbox{a.e. in } \Omega,
       \end{align}
       where ${^{\pm}}{I}{^{\alpha}}$ denote the right/left fractional integral operators
       (cf. \cite{Samko, Feng_Sutton}) and 
\[          \kappa^{\alpha}_{-}(x) := (x-a)^{\alpha -1}, \quad \kappa^{\alpha}_{+}(x) := (b-x)^{\alpha-1} ;\quad  c_{-}^{\sigma}  := \frac{  {^{-}}{I}{^{\sigma}} f(a) }{\Gamma(\sigma)},\quad
c_{+}^{\sigma} := \frac{   {^{+}}{I}{^{\sigma}} f(b) }{\Gamma(\sigma)}.
\]
  
   \end{theorem}

\section{New Families of Fractional Sobolev Spaces}\label{sec-3}
With weak fractional derivatives in hand, it is natural 
to define fractional Sobolev spaces in the same manner as in the integer order case. 
The goal of this section is exactly to introduce new families of Sobolev spaces
based on such an approach.

\subsection{Definitions of New Fractional Sobolev Spaces}\label{sec-3.1}
We now introduce our fractional Sobolev spaces using weak fractional derivatives as follows. 

    \begin{definition}
       For $\alpha>0$, let $m :=[\alpha]$. For $1 \leq p \leq \infty$,  the left/right fractional Sobolev spaces ${^{\pm}}{W}{^{ \alpha , p}} (\Omega)$ are defined by  
        \begin{align} \label{FSS}
             {^{\pm}}{W}{^{ \alpha , p}} (\Omega) = \left\{ u \in W^{m,p}(\Omega): 
             {^{\pm}}{\mathcal{D}}{^{\alpha}}   u \in L^{p}(\Omega) \right\},
        \end{align}
        which are endowed respectively with the norms 
        \begin{align} \label{FSS_norm}
          \|u\|_{{^{\pm}}{W}{^{\alpha , p}}(\Omega)}:= \begin{cases}
          \Bigl(\left\|u\right\|_{W^{m,p}(\Omega)}^{p} + \left\|{^{\pm}}{\mathcal{D}}{^{\alpha}}  u \right\|_{L^{p}(\Omega)}^{p} \Bigr)^{\frac{1}{p}} &\qquad \text{if } 1 \leq p < \infty,\\
          \|u\|_{W^{m,\infty}(\Omega)} 
          + \left\|{^{\pm}}{\mathcal{D}}{^{\alpha}}  u \right\|_{L^{\infty}(\Omega)} &\qquad \text{if } p = \infty.
          \end{cases} 
        \end{align}

    \end{definition}

    \begin{definition}
        For $  \alpha >0$ and $1 \leq p \leq \infty$, the symmetric fractional Sobolev space is defined by 
        \begin{align}\label{SFSS}
           { {W}^{\alpha,p}(\Omega)}:= {^{-}}{W}{^{\alpha,p}}(\Omega) \cap {^{+}}{W}{^{\alpha,p}}(\Omega),
        \end{align}
        which is endowed with the norm
        \begin{align}\label{SFSS_norm}
            \|u\|_{{ {W}^{\alpha,p}(\Omega)} } &:= \begin{cases} \displaystyle{
            \Bigl( \|u\|_{{^{-}}{W}{^{\alpha,p}}(\Omega)}^{p} + \|u\|_{{^{+}}{W}{^{\alpha,p}}(\Omega)}^{p} \Bigr)^{\frac{1}{p}} }  &\qquad \text{if } 1\leq p < \infty,\\
            \displaystyle{ \|u\|_{{^{-}}{W}{^{\alpha,\infty}}(\Omega)} + \|u \|_{{^{+}}{W}{^{\alpha,\infty}}(\Omega)} } &\qquad \text{if } p = \infty.
             \end{cases}
        \end{align}
    \end{definition}

\begin{remark}
	For $\alpha>0$, let $m :=[\alpha]$ and $\sigma:=\alpha- m$. Using the semigroup property of  weak fractional derivatives, it is easy to see that 
	\begin{align} \label{FSS_equiv}
	{^{\pm}}{W}{^{ \alpha , p}} (\Omega) = \left\{ u \in W^{m,p}(\Omega): 
	\mathcal{D}^m ({^{\pm}}{\mathcal{D}}{^{\sigma}} u) \in L^{p}(\Omega) \right\}
	\end{align}
	and
	\begin{align} \label{FSS_norm_equiv}
	\|u\|_{{^{\pm}}{W}{^{\alpha , p}}(\Omega)}:= \begin{cases}
	\Bigl(\left\|u\right\|_{W^{m,p}(\Omega)}^{p} + \left\|\mathcal{D}^m ({^{\pm}}{\mathcal{D}}{^{\sigma}}   u) \right\|_{L^{p}(\Omega)}^{p} \Bigr)^{\frac{1}{p}} &\qquad \text{if } 1 \leq p < \infty,\\
	\|u\|_{W^{m,\infty}(\Omega)} 
	+ \left\|\mathcal{D}^m ({^{\pm}}{\mathcal{D}}{^{\sigma}}  u) \right\|_{L^{\infty}(\Omega)} &\qquad \text{if } p = \infty.
	\end{cases}
	\end{align}
\end{remark}

\subsection{Elementary Properties of New Fractional Sobolev Spaces}\label{sec-3.2}
  Below we gather several basic properties of the newly defined fractional Sobolev spaces. Since
  their proofs are straightforward, we omit them to save space and refer the reader to
  \cite[Section 4]{Feng_Sutton} for the details. 
    
     \begin{proposition}
        Let $ \alpha >0$, $1 \leq p \leq \infty$, and $\Omega \subseteq \R$. Then $\left\| \cdot \right\|_{{^{\pm}}{W}{^{\alpha , p}}(\Omega)}$ are  norms  on ${^{\pm}}{W}{^{\alpha , p}}(\Omega)$,  which are in turn Banach spaces with these norms. Consequently, 
        $W{^{\alpha , p}(\Omega)}$ is also a Banach space.  Moreover,  ${^{\pm}}{W}{^{\alpha ,2}}(\Omega)$ are Hilbert spaces with inner products 
         \[ 
        \langle u,v \rangle_\pm : =(u,v) + \bigl( {^{\pm}}{\mathcal{D}}{^{\alpha}} u , {^{\pm}}{\mathcal{D}}{^{\alpha}}v \bigr)
        = \int_{\Omega} uv\,dx + \int_{\Omega} {^{\pm}}{\mathcal{D}}{^{\alpha}} u {^{\pm}}{\mathcal{D}}{^{\alpha}}v\,dx. 
        \]
        We write ${^{\pm}}{H}{^{\alpha}}(\Omega) : = {^{\pm}}{W}{^{\alpha,2}}(\Omega)$   
        and $ {H}^{\alpha}(\Omega): =  {W}^{\alpha,2}(\Omega)$.
    \end{proposition}

    \begin{proposition} ${^{\pm}}{W}{^{\alpha,p}}(\Omega)$ is reflexive for $1 < p < \infty$ and separable for $1 \leq p < \infty$. Consequently, the same assertions hold for $ {W}^{\alpha,p}(\Omega)$.
    \end{proposition}

\section{Advanced Properties of New Fractional Sobolev Spaces}\label{sec-4}

\subsection{Approximation and Characterization of New Fractional Sobolev Spaces}\label{sec-4.1}
    
    In the integer order case, an alternative way to define Sobolev spaces is 
    to use the completion spaces of smooth functions under chosen Sobolev norms. 
    The goal of this subsection is to establish an analogous result for fractional Sobolev spaces
    introduced in Section \ref{sec-3.1}. To the end,  we first need to introduce spaces 
    that we refer to as \textit{one-side supported spaces}.
    For $(a,b) \subseteq \R$, we set 
    	\begin{align*}
    	{^{-}}{C}{^{\infty}_{0}}((a,b)) &: = \{ \varphi \in C^{\infty}((a,b)) \, | \,\exists\, c \in (a,b) \mbox{ such that } \varphi(x) \equiv 0\,\, \forall x > c\},\\
    	{^{+}}{C}{^{\infty}_{0}}((a,b)) &: = \{ \varphi \in C^{\infty}((a,b)) \, | \,\exists\, c \in (a,b)  \mbox{ such that }  \varphi(x) \equiv 0  \,\,\forall x < c\}.
    	\end{align*}
    
        Here we use the notation ${^{-}}{C}{^{\infty}_{0}}((a,b))$ to represent functions whose support is not actually a compact subset of $(a,b)$. In particular, if $u \in {^{-}}{C}{^{\infty}_{0}}((a,b))$, then $\mbox{supp}(u) = [a,c]$, which is not a compact subset of $(a,b)$. 
        The use of ${^{-}}{C}{^{\infty}_{0}}$ and ${^{+}}{C}{^{\infty}_{0}}$ (particularly the direction indication) are chosen so that these spaces will pair with the appropriate direction-dependent spaces ${^{-}}{W}{^{\alpha,p}}$ and ${^{+}}{W}{^{\alpha,p}}$ respectively. The need for these and the aforementioned space coupling will become 
        evident in Section \ref{sec-4.3}.
    
    We now introduce completion spaces using the norms defined in Section \ref{sec-3.1}.
    
    \begin{definition}\label{completionspaces}
    	Let $ \alpha >0$ and $1 \leq p \leq \infty$.  We define 
    	\begin{itemize}
    		\item[(i)] ${^{\pm}}{\overline{W}}{^{\alpha,p}}(\Omega)$ to be the completion of $C^{\infty}(\Omega)$ in the norm $\|\cdot\|_{{^{\pm}}{W}{^{\alpha,p}}(\Omega)}$,
    		\item[(ii)] ${^{\pm}}{\overline{W}}{^{\alpha,p}_{0}}(\Omega)$ to be  the completion of ${^{\pm}}{C}{^{\infty}_{0}}(\Omega)$ in the norm $\|\cdot \|_{{^{\pm}}{W}{^{\alpha,p}}(\Omega)}$,
    		\item[(iii)] $ {\overline{ {W}}^{\alpha,p}(\Omega)}$ to be the completion of $C^{\infty}(\Omega)$ in the  norm $\|\cdot\|_{{ {W}^{\alpha,p}(\Omega)}}$,
    		\item[(iv)] $\overline{{W}}^{\alpha,p}_{0}(\Omega)$ to be  the completion of $C^{\infty}_{0}(\Omega)$ in the  norm $\|\cdot\|_{ {W}^{\alpha,p}(\Omega)}$.
    	\end{itemize}
    \end{definition}

    
        \subsubsection{\bf The Finite Domain Case: $\Omega=(a,b)$}\label{sec-4.1.1}
        The goal of this subsection is to establish the equivalence  
         ${^{\pm}}{\overline{W}}{^{\alpha,p}}(\Omega) = {^{\pm}}{W}{^{\alpha, p}} (\Omega)$. 
         This is analogous to Meyers and Serrin's celebrated  ``$H = W$" result (cf. \cite{Adams, Evans, Meyers}). 
         It turns out that the proof is more complicated 
         due to more complicated product rule for fractional derivatives.

    \begin{lemma}
        Let $ \alpha >0$ and $1 \leq p <\infty$. Suppose  $\psi \in C^{\infty}_{0}(\Omega)$ and $u \in {^{\pm}}{W}{^{\alpha,p}}(\Omega)$. Then $u \psi \in {^{\pm}}{W}{^{\alpha,p}}(\Omega)$.
    \end{lemma}

    \begin{proof}
    	We only give a proof for $0<\alpha<1$ because the case $\alpha>1$ follows immediately by setting
    	$m:=[\alpha]$ and $\sigma:=\alpha-m$ and using the Meyers and Serrin's celebrated result.
    	
        Since $\psi \in C^{\infty}_{0}(\Omega)$, there exists a compact set $K:=\supp(\psi) \subset \Omega$ such that  $\psi \in C^{\infty}(K)$. Then there exists $0 \leq M <\infty$ so that $M_0 = \max_{\Omega}|\psi|$ and 
        $\|\psi\|_{L^{\infty}(\Omega)} = M_{0} < \infty.$      
        Since $u \in L^{p}(\Omega)$, then trivially we have $u\psi \in L^{p}(\Omega)$.

        It remains to show that ${^{\pm}}{\mathcal{D}}{^{\alpha}} (u\psi) \in L^{p}(\Omega)$. 
        To that end, by \cite[Theorem 4.3]{Feng_Sutton} for arbitrarily large $m \in\N$, we get
        \begin{align*}
            &\left\|{^{\pm}}{\mathcal{D}}{^{\alpha}} (u \psi) \right\|_{L^{p}(\Omega)} \leq \left\|{^{\pm}}{\mathcal{D}}{^{\alpha}} u \cdot \psi \right\|_{L^{p}(\Omega)} 
            + \left\|\sum_{k=1}^{m} C_{k} {^{\pm}}{I}{^{k-\alpha}} u D^{k} \psi  + {^{\pm}}{R}{^{\alpha}_{m}}(u,\psi) \right\|_{L^{p}(\Omega)} \\
            &\,\, \leq M_0 \left\| {^{\pm}}{\mathcal{D}}{^{\alpha}} u \right\|_{L^{p}(\Omega)} + M_1 \sum_{k=1}^{m}\left| C_{k}\right| \left\|{^{\pm}}{I}{^{k - \alpha}} u\right\|_{L^{p}(\Omega)}  + \left\| {^{\pm}}{R}{^{\alpha}_{m}} (u,\psi) \right\|_{L^{p}(\Omega)}\\
            &\,\, \leq M_0 \left\| {^{\pm}}{\mathcal{D}}{^{\alpha}} u \right\|_{L^{p}(\Omega)} + M_1 \sum_{k=1}^{m} \dfrac{\left| C_{k}\right| \cdot |\Omega|^{k-\alpha} }{ ( k- \alpha) \Gamma(k - \alpha)} \left\|u\right\|_{L^{p}(\Omega)}  + \left\| {^{\pm}}{R}{^{\alpha}_{m} }(u,\psi) \right\|_{L^{p}(\Omega)},
        \end{align*}
        where $M_1 : = \sup \left| D^{k} \psi(x) \right|$ taken over $1 \leq k \leq m$ and $x \in \Omega$. Clearly, $M_1 < \infty$ since $\psi \in C^{\infty}_{0} (\Omega)$. Because  $u, {^{\pm}}{\mathcal{D}}{^{\alpha}} u\in L^{p}(\Omega)$  and 
        \begin{align*}
            \dfrac{\left| C_{k}\right| \cdot |\Omega|^{k-\alpha} }{ ( k- \alpha) \Gamma(k - \alpha)} &= \dfrac{\Gamma(1 + \alpha) |\Omega|^{k-\alpha}  }{(k - \alpha) \Gamma(k+1) |\Gamma( 1- k + \alpha)|} <\infty,
        \end{align*}
        the first two terms on the right-hand side of the above inequality are finite. 

        It is left to show that the remainder term is also finite in $L^{p}(\Omega)$. To be precise, we consider the case for ${^{-}}{R}{^{\alpha}_{m}}(u,\psi)$. By its definition we get  
        \begin{align*}
            \left|{^{-}}{R}{^{\alpha}_{m}}(u,\psi)(x)\right| 
            &\leq \dfrac{M_2}{m! |\Gamma(- \alpha)|} \int_{a}^{x} \int_{y}^{x} \dfrac{|u(y)|}{(x-y)^{1+\alpha}} (x-z)^{m} \, dz dy\\
            &= \dfrac{M_2}{(m+1)! |\Gamma(- \alpha)|} {^{-}}{I}{^{ m - \alpha+1}}|u|(x)
        \end{align*}
        where $M_2 : = \sup_{x \in \Omega} \left| \psi^{(m+1)} (x) \right|$. Since $\Gamma( - \alpha) \neq 0$, hence the 
        coefficient is finite. The same estimate holds for ${^{+}}{R}{^{\alpha}_{m}}(u,\psi)$ as well. Thus,
        \begin{align*}
            \left\|{^{\pm}}{R}{^{\alpha}_{m}}(u,\psi) \right\|_{L^{p}(\Omega)} 
            &\leq \left\| \dfrac{M_3}{(m+1)! |\Gamma(- \alpha)|} {^{\pm}}{I}{^{ m - \alpha+1}}|u| \right\|_{L^{p}(\Omega)} \\ 
            &\leq \dfrac{M_2|\Omega|^{m - \alpha +1}}{(m+1)! (m- \alpha +1) |\Gamma(- \alpha)|  \Gamma(m - \alpha +1)} \left\| u \right\|_{L^{p}(\Omega)} < \infty .
        \end{align*}
        This concludes that ${^{\pm}}{\mathcal{D}}{^{\alpha}} (u \psi) \in L^{p}(\Omega)$, consequently, $u \psi \in {^{\pm}}{W}{^{\alpha,p}}(\Omega)$. 
    \end{proof}

We are now ready to state and prove the following fractional counterpart of Meyers and Serrin's  
``$H = W$" result.

    \begin{theorem}\label{H=W}
        Let $ \alpha >0$ and $1\leq p <\infty$. Then ${^{\pm}}{\overline{W}}{^{\alpha,p}}(\Omega) = {^{\pm}}{W}{^{\alpha,p}}(\Omega)$.
    \end{theorem}

    \begin{proof}
        Because ${^{\pm}}{W}{^{\alpha,p}}(\Omega)$ are Banach spaces, then by the definition we have  ${^{\pm}}{\overline{W}}{^{\alpha,p}}(\Omega) \subseteq {^{\pm}}{W}{^{\alpha,p}}(\Omega)$. To show 
        the reverse inclusion ${^{\pm}}{\overline{W}}{^{\alpha,p}}(\Omega) \supseteq {^{\pm}}{W}{^{\alpha,p}}(\Omega)$, it suffices to show that $C^{\infty}(\Omega)$ is dense in ${^{\pm}}{W}{^{\alpha , p}}(\Omega)$.  This will be done in the same fashion as in the 
        integer order case given in \cite{Meyers} (also see \cite{Adams, Evans}). Below we shall only give a proof for 
        the case $0<\alpha< 1$ because the case $\alpha >1$ follows similarly.
        
        For $k = 1,2...$ let 
        \begin{align*}
            \Omega_{k} = \left\{x \in \Omega : |x| < k \text{ and } \text{dist}(x , \partial \Omega) > \frac{1}{k} \right\}.
        \end{align*}
        For convenience, let $\Omega_{-1} = \Omega _0 = \emptyset$. Then 
        \begin{align*}
            \Theta = \left\{ \Omega'_{k} : \Omega'_{k} = \Omega_{k+1} \setminus \overline{\Omega}_{k-1} \right\}
        \end{align*}
        is an open cover of $\Omega$. Let $\{\psi_k\}_{k =1}^{\infty}$ be a $C^{\infty}$-partition of unity of $\Omega$ subordinate to $\Theta$ so that  $\supp\left( \psi_{k} \right) \subset \Omega'_{k}$. Then $\psi_{k} \in C^{\infty}_{0} \left(\Omega'_{k} \right)$. 

        If $ 0 < \eps < \frac{1}{(k+1)(k+2)}$, let $\eta_{\eps}$ be a $C^{\infty}$ mollifier satisfying
        \begin{align*}
            \supp \left( \eta_{\eps}\right) \subset \left\{x : |x| < \frac{1}{(k+1)(k+2)}\right\}.
        \end{align*}
        Evidently,  $\eta_{\eps} * \left( \psi_{k} u \right)$ has support in $\Omega_{k+2} \setminus \overline{\Omega}_{k-2} \subset \subset \Omega$. Since $\psi_{k} u \in {^{\pm}}{W}{^{\alpha,p}}(\Omega)$ we can choose $0 < \eps_{k} < \frac{1}{(k+1)(k+2)}$ such that 
        \begin{align*}
            \left\|\eta_{\eps_{k}} * (\psi_ku) - \psi_{k} u \right\|_{{^{\pm}}{W}{^{\alpha,p}}(\Omega)} < \dfrac{\eps}{2^{k}}.
        \end{align*}
        Let $v = \sum_{k=1}^{\infty} \eta_{\eps_{k}} * (\psi_{k} u)$. On any $U \subset \subset \Omega$ only finitely
         many terms in the sum can fail to vanish. Thus, $v \in C^{\infty}(\Omega)$. For $x \in \Omega_{k}$ we have 
        \begin{align*}
            u(x) = \sum_{j =1}^{k+2} (\psi_{j} u)(x), \qquad v(x) = \sum_{j=1}^{k+2} \left(\eta_{\eps_{j}} * (\psi_{j}u)\right)(x). 
        \end{align*}
        Therefore, 
        \begin{align*}
            \|u - v\|_{{^{\pm}}{W}{^{\alpha, p}}\left(\Omega_k\right)} &= \biggl\|\sum_{j=1}^{k+2} \eta_{\eps_j}*(\psi_{j}u) - \psi_{j} u \biggr\|_{{^{\pm}}{W}{^{\alpha , p}}(\Omega_{k})} \\
            &\leq \sum_{j=1}^{k+2} \left\|\eta_{\eps_j}* (\psi_{j}u) - \psi_{j} u \right\|_{{^{\pm}}{W}{^{\alpha , p}}(\Omega)} < \eps < \frac{1}{(k+1)(k+2)}.
        \end{align*}
        Setting $k \rightarrow \infty$ and applying the Monotone Convergence theorem yields the desired result 
        $\left\|u - v\right\|_{{^{\pm}}{W}{^{\alpha,p}}(\Omega)} < \eps$. The proof is complete.
    \end{proof}
    
    One crucial difference between integer order Sobolev spaces $W^{k,p}(\Omega)$ and fractional order Sobolev spaces $^{{\pm}}{W}{^{\alpha,p}}(\Omega)$ (for $0<\alpha<1$) is that
    piecewise constant functions are not dense in the former, but are dense in the latter  
    (see the next theorem below). 
    Such a difference helps characterize a major difference between the fractional order weak derivatives and 
    integer order weak derivatives.
    
    %
    %

    \begin{theorem}
        Let $\Omega=(a,b)$,  $ \alpha >0 $ and $1 \leq p <\infty$ so that $\alpha p  <1$. Then piecewise constant functions are dense in $^{{\pm}}{W}{^{\alpha,p}}(\Omega)$.
    \end{theorem}

    \begin{proof}
        Let $\eps > 0$ and $u \in {^{\pm}}{W}{^{\alpha,p}}((a,b))$ for $0<\alpha <1$.  The case when $\alpha >1$ follows as a direct consequence of the Riemann-Liouville derivative definition and the calculations below. Since $C^{\infty}((a,b))$ is dense in ${^{\pm}}{W}{^{\alpha,p}}((a,b))$, then there exists $v \in C^{\infty}((a,b))$ such that $\|u-v\|_{{^{\pm}}{W}{^{\alpha,p}}((a,b))} < \frac{\eps}{2}$. Moreover, choose a piecewise constant 
        function $w$ such that $\sup_{x \in (a,b)} |v(x) - w(x)| < \frac{\eps}{2}\max\{|b-a|^{1-\alpha p} , |b-a|\} =:M$.
        Then 
        \begin{align*}
            \|u - w\|_{L^{p}((a,b))}^{p} &\leq \|u-v\|_{L^{p}((a,b))}^{p} + \|v - w\|_{L^{p}((a,b))}^{p} \\ 
            &< \dfrac{\eps}{2} + \int_{a}^{b} |v-w|^{p}\,dx 
            <\dfrac{\eps}{2} + \left( \dfrac{\eps}{M}\right)^{p} |b-a| 
            \leq  \eps .
        \end{align*}
        Similarly,  on noting that ${^{\pm}}{\mathcal{D}}{^{\alpha}}w$ exists and belongs to $L^p((a,b))$, 
        we have
        \begin{align*}
            \left\|{^{\pm}}{\mathcal{D}}{^{\alpha}} u - {^{\pm}}{\mathcal{D}}{^{\alpha}}w \right\|_{L^{p}((a,b))}^{p} &\leq \left\|{^{\pm}}{\mathcal{D}}{^{\alpha}} u - {^{\pm}}{\mathcal{D}}{^{\alpha}}v \right\|_{L^{p}((a,b))}^{p} + \left\|{^{\pm}}{\mathcal{D}}{^{\alpha}} v - {^{\pm}}{\mathcal{D}}{^{\alpha}}w \right\|_{L^{p}((a,b))}^{p}\\
            &< \dfrac{\eps}{2} + \left\|{^{\pm}}{D}{^{\alpha}} v - {^{\pm}}{D}{^{\alpha}}w \right\|_{L^{p}((a,b))}^{p},
        \end{align*}
        and the last term can be bounded as follows 
        \begin{align*}
            \left\|{^{\pm}}{D}{^{\alpha}} v - {^{\pm}}{D}{^{\alpha}}w \right\|_{L^{p}((a,b))}^{p}&= \int_{a}^{b} \left| \dfrac{d}{dx} \int_{a}^{x} \dfrac{v(y) - w(y)}{(x-y)^{\alpha}}\,dy \right|^{p}\,dx\\
            &< \left(\dfrac{\eps}{2 M}\right)^{p} \int_{a}^{b} \dfrac{dx}{(x-a)^{\alpha p}} 
            <\dfrac{\eps}{2}.
        \end{align*}
        This proves the assertion. 
    \end{proof}

    \subsubsection{\bf Infinite Domain Case: $\Omega=\R$}\label{sec-4.1.2}
    The approximation of functions in the fractional Sobolev functions on $\R$ is much easier than the case when $\Omega=(a,b)$. In this case, all of the legwork has already been done in the characterization theorem for weak derivatives  (see Theorem \ref{characterization}).
    
%

    \begin{theorem}
        Let $\alpha >0$ and $1 \leq p <\infty$. Then $C^{\infty}_{0}(\R)$ is dense in ${^{\pm}}{W}{^{\alpha,p}}(\R)$. Hence, ${^{\pm}}{\overline{W}}{^{\alpha,p}}(\R) =
        {^{\pm}}{\overline{W}}{^{\alpha,p}_{0}}(\R)= {^{\pm}}{W}{^{\alpha,p}}(\R)$.
    \end{theorem}
    
    \begin{proof}
        We only give a proof for the case $0<\alpha<1$ since the case $\alpha>1$ follows similarly. Let $u \in {^{\pm}}{W}{^{\alpha,p}}(\R)$. Recall that there exists a sequence $\left\{u_j \right\}_{j = 1}^{\infty} \subset C^{\infty}_{0}(\R)$ such that $u_j \rightarrow u$ in $L^{p}(\R)$ and ${^{\pm}}{\mathcal{D}}{^{\alpha}} u_j \rightarrow {^{\pm}}{\mathcal{D}}{^{\alpha}} u$ in $L^{p}(\R)$ because ${^{\pm}}{\mathcal{D}}{^{\alpha}} u \in L^{p}(\R)$. The proof is complete. 
    \end{proof}

    \subsection{Extension Operators}\label{sec-4.2}
    In this subsection we address the issue of extending Sobolev functions from a finite domain $\Omega=(a,b)$ to 
    the real line $\R$. As we shall see below, constructing such an extension operator in ${^{\pm}}{W}{^{\alpha,p}}(\Omega)$ requires a different process and added conditions relative to the integer order case. Recall that spaces ${^{\pm}}{W}{^{\alpha,p}}(\Omega)$ differ greatly from 
    integer Sobolev spaces due to the following characteristics: {\em 
    (i) ${^{\pm}}{W}{^{\alpha,p}}$ is direction-dependent  and domain-dependent;
    (ii) fractionally differentiable functions inherit singular kernel functions;
    (iii) continuity is not a necessary condition for fractional differentiability;
    (iv) compact support is a desirable property to dampen the singular effect of the kernel functions and pollution.}
    Moreover, we also note that due to the pollution effect 
    of fractional derivatives,  zero function values may result in nonzero 
    contribution to fractional derivatives, 
    controlling the pollution contributions is also the key in the subsequent analysis.  
   
 \subsubsection{\bf Extensions of Compactly Supported Functions}\label{sec-4.2.1}
    We first consider the easy case of compactly supported functions. In this case, we show that 
    the trivial extension will do the job. 
    
    \begin{lemma}\label{TrivialExtension}
        Let $\Omega=(a,b)$, $\alpha>0$, and $1 \leq p < \infty$. If $u \in {^{\pm}}{W}{^{\alpha,p}}(\Omega)$ and $K:=\supp(u) \subset\subset\Omega$, then the trivial extension $\tilde{u}$ belongs to ${^{\pm}}{W}{^{\alpha,p}}(\R)$ and there exists 
        $C = C(\alpha, p,K)>0$ such that
        \begin{align*}
            \|\tilde{u}\|_{{^{\pm}}{W}{^{\alpha,p}}(\R)} \leq C \left\|u\right\|_{{^{\pm}}{W}{^{\alpha, p}}(\Omega)}.
        \end{align*}
    \end{lemma}
 
    \begin{proof}
        Let $\{u_j\}_{j = 1}^{\infty}  \subset C^{\infty}_{0}(\Omega)$ be an approximating sequence of $u$ and define
        \begin{align*}
            \tilde{u}_j(x):= \begin{cases}
             u_j(x) &\text{if } x\in \Omega, \\ 
             0 &\text{if } x \in \R \setminus \Omega .
             \end{cases}
         \end{align*}
         Clearly, $\|\tilde{u}_j\|_{L^{p}(\R)} = \|u_j\|_{L^{p}(\Omega)} <\infty$. Let $\varphi \in C^{\infty}_{0}(\R)$ and by the integration by parts formula for classical fractional derivatives (cf. \cite[Theorem 2.5]{Feng_Sutton}), we get 
        \begin{align*}
            \int_{\R} \tilde{u}_j {^{\mp}}{D}{^{\alpha}}\varphi  &= \int_{\R} {^{\pm}}{D}{^{\alpha}} \tilde{u}_j \varphi .
        \end{align*}
        For clarity, let $\supp(u_j) \subset K\subseteq (c,d) \subset\subset (a,b)$ and we look at the left derivative.
        \begin{align*}
            \left\|{^{-}}{\mathcal{D}}{^{\alpha}} \tilde{u}_j \right\|_{L^{p}(\R)}^{p} 
            &= \left\| {^{-}}{D}{^{\alpha}} u_j \right\|_{L^{p}((a,b))}^{p} + \left\|L(u_j)\right\|_{L^{p}((b,\infty))}^{p} \\ 
            &= \left\|{^{-}}{\mathcal{D}}{^{\alpha}}u_j \right\|_{L^{p}((a,b))}^{p} + \int_{b}^{\infty} \biggl| \int_{c}^{d} \dfrac{u_j(y)}{(x-y)^{1+\alpha}} \,dy \biggr|^{p}\,dx \\ 
            &\leq \left\|{^{-}}{\mathcal{D}}{^{\alpha}}u_j \right\|_{L^{p}((a,b))}^{p} + \|u_j\|_{L^{p}((a,b))}^{p}\dfrac{(b-a)^{\frac{p}{q}}}{(d-b)^{p(1+\alpha) - 1}}.
        \end{align*}
        Then there exists $C= C(\alpha,p,K)$ such that 
        \begin{align*}
            \left\|\tilde{u}_j\right\|_{{^{\pm}}{W}{^{\alpha,p}}(\R)} \leq C \left\|u_j\right\|_{{^{\pm}}{W}{^{\alpha,p}}(\Omega)}.
        \end{align*}
        Now, we need to show that the appropriate limits are realized. By construction, $\left\|u_j \right\|_{{^{\pm}}{W}{^{\alpha,p}}(\Omega)} \rightarrow \|u\|_{{^{\pm}}{W}{^{\alpha,p}}(\Omega)}.$ Therefore, $\lim_{j\rightarrow \infty} \left\|\tilde{u}_j\right\|_{{^{\pm}}{W}{^{\alpha,p}}(\R)} < \infty$. Let $\eps > 0$ and choose sufficiently large $m$ and $n$ so that
        \begin{align*}
            \left\|\tilde{u}_m - \tilde{u}_n\right\|_{{^{\pm}}{W}{^{\alpha,p}}(\R)} \leq C \left\|u_n - u_m \right\|_{{^{\pm}}{W}{^{\alpha,p}}(\Omega)} < \eps.
        \end{align*}
        Hence, $\left\{\tilde{u}_j\right\}_{j=1}^{\infty}$ is a Cauchy sequence in ${^{\pm}}{W}{^{\alpha,p}}(\R)$. Since ${^{\pm}}{W}{^{\alpha,p}}(\R)$ is complete, there exists $v\in {^{\pm}}{W}{^{\alpha,p}}(\R)$ such that $u_j \rightarrow v$ in ${^{\pm}}{W}{^{\alpha,p}}(\R)$. We claim finally that $v = \tilde{u}$ almost everywhere. For sufficiently large $j$, we have that
        \begin{align*}
            \left\|\tilde{u} - v \right\|_{L^{p}(\R)}&\leq \left\|\tilde{u} - \tilde{u}_j\right\|_{L^{p}(\R)}+\left\|\tilde{u}_j - v\right\|_{L^{p}(\R)}\\
            &= \left\|u - u_j \right\|_{L^{p}(\Omega)} + \left\|\tilde{u}_j - v\right\|_{L^{p}(\R)} 
            < \eps. 
        \end{align*}
        Therefore, $v= \tilde{u}$ almost everywhere in $\R$. 
        This concludes that the trivial extension $\tilde{u}$ satisfies the desired properties on compactly supported functions.
    \end{proof}
    
    \begin{corollary}
         The same result holds for any $u \in  {W}^{\alpha,p}(\Omega)$ with the same construction. 
    \end{corollary}
    

    \subsubsection{\bf Interior Extensions}\label{sec-4.2.2}
    For any function $u \in {^{\pm}}{W}{^{\alpha,p}}(\Omega)$ and $\Omega' \subset\subset \Omega$, 
    we first rearrange  $u$ in  $\Omega\setminus \Omega'$ so that the rearranged function $u^*$ has compact support in $\Omega$ and coincide with $u$ in $\Omega'$.  With the help of such a rearrangement and the
     extension  result of the previous subsection we then can extend any function $u$ in 
     $ {^{\pm}}{W}{^{\alpha,p}}(\Omega)$ to a function in ${^{\pm}}{W}{^{\alpha,p}}(\R)$ with some 
     preferred properties. We refer to such an extension as an {\em interior extension} of $u$. 
    
    \begin{lemma}\label{InteriorExtesionLemma}
    	Let $\Omega=(a,b)$ and $\alpha>0$. 
        For each $\Omega' \subset \subset  \Omega$, there exists a compact set $K\subset \Omega$ 
        and a constant $C=C(\alpha, K) >0$, such that for every $u \in {^{\pm}}{W}{^{\alpha,p}}(\Omega)$, 
        there exists $u^* \in {^{\pm}}{W}{^{\alpha,p}}(\Omega)$ satisfying 
        \begin{enumerate}
            \item[{\rm (i)}] $u^* = u$   a.e. in $\Omega'$,
            \item[{\rm (ii)}] $\supp(u^*) \subseteq K$,
            \item[{\rm (iii)}] $\|u^*\|_{{^{\pm}}{W}{^{\alpha,p}}(\Omega)} 
            \leq C \|u\|_{{^{\pm}}{W}{^{\alpha,p}}(\Omega)}.$
        \end{enumerate}
    \end{lemma}
    
    \begin{proof}
    	Again, we only give a proof for the case $0<\alpha<1$ because the case $\alpha>1$ can be proved similarly.
        Choose $\Omega' \subset \subset \Omega$. Let $\{B_{i}\}_{i=1}^{N}$ be a cover of $\overline{\Omega'}$ 
        with a subordinate partition of unity $\{ \psi_{i}\}_{i=1}^{N} \subset C^{\infty}(\Omega)$ in the sense that $\supp(\psi_i)\subset B_i$ for $i=1,2,\cdots, N$. Define 
        $u^*:\Omega \rightarrow \R$ by $u^* := u \psi$ with $\psi:= \sum_{i=1}^{N} \psi_{i}$. 
        Note that $u^* = u$ almost 
        everywhere on $\Omega'$ and $\supp(u^*) \subseteq K:=\cup \overline{B}_i$. We need to show 
        that $u^* \in {^{\pm}}{W}{^{\alpha,p}}(\Omega)$. Of course, 
        \begin{align*}
            \|u^*\|_{L^{p}(\Omega)} &= \bigl\| u \psi  \bigr\|_{L^{p}(\Omega)} 
            = \bigl\| u \psi \bigr\|_{L^p(K)} \leq \| u \|_{L^{p}(K)} \leq \|u\|_{L^{p}(\Omega)}.
        \end{align*}
        Next, by the product rule for weak fractional derivatives \cite[Theorem 4.3]{Feng_Sutton}, 
        we get 
        \begin{align*}
            {^{\pm}}{\mathcal{D}}{^{\alpha}} u^* &= {^{\pm}}{\mathcal{D}}{^{\alpha}} u \cdot 
           \psi + \sum_{k=1}^{m} C(k,\alpha) {^{\pm}}{I}{^{k-\alpha}} u  D^{k} \psi  + {^{\pm}}{R}{^{\alpha}_{m}} (u,\psi).
        \end{align*}
        It follows by direct calculations that 
        \begin{align*}
            \|{^{\pm}}{\mathcal{D}}{^{\alpha}} u^*\|_{L^{p}(\Omega)} \leq C \|u\|_{{^{\pm}}{W}{^{\alpha,p}}(\Omega)}.
        \end{align*}
        Hence, $u^* \in {^{\pm}}{W}{^{\alpha,p}}(\Omega)$ and there exists $C >0$ 
        such that assertion (iii) holds. The proof is complete. 
    \end{proof}

Now, we are ready to state the following interior extension theorem. 

    \begin{theorem}\label{InteriorExtension}
    	Let $\Omega=(a,b)$ and $\alpha>0$. 
        For each $\Omega' \subset\subset \Omega$,   there exist  a compact set $K\subset \Omega$  and  a constant $C = C(\alpha ,p, K) >0$ so that for every $u \in {^{\pm}}{W}{^{\alpha,p}}(\Omega)$,  
        there exists a mapping $E: {^{\pm}}{W}{^{\alpha,p}}(\Omega) \rightarrow {^{\pm}}{W}{^{\alpha,p}}(\R)$ so that 
        \begin{enumerate}
            \item[{\rm (i)}]  ${Eu} = u$  a.e. in $\Omega'$,
            \item[{\rm (ii)}] $\supp(E{u})\subseteq K,$
            \item[{\rm (iii)}] $\|E{u}\|_{{^{\pm}}{W}{^{\alpha,p}}(\R)} \leq C \|u\|_{{^{\pm}}{W}{^{\alpha,p}}(\Omega)}.$
        \end{enumerate}
        We call $Eu$ an (interior) extension of $u$ to $\R$. 
        
    \end{theorem}
 
    \begin{proof}
        For $u\in {^{\pm}}{W}{^{\alpha,p}}(\Omega) $, let $u^* \in {^{\pm}}{W}{^{\alpha,p}}(\Omega) $ denote the rearrangement of $u$  as defined in Lemma \ref{InteriorExtesionLemma},  let  $K\subset\subset \Omega$ and $C(\alpha, K)$ be the same as well. Since $u^*$ has a  compact support in $\Omega$,  
         we can invoke Lemma \ref{TrivialExtension}  to conclude that  $E{u} :=
        \widetilde{u^*}$ satisfies the desired properties (i)--(iii) with $C=C(\alpha,p,K) C(\alpha, K)$. The proof is complete. 
    \end{proof}

    %
    
  \begin{remark}
  	We emphasize that  the extension operator $E$ defined above depends on the choice of subdomain $\Omega'$, 
  	on the other hand, it does not depend on the left or right direction,  consequently,  $E$ also provides a valid 
  	{\em interior extension operator} from the symmetric space $W^{\alpha ,p}(\Omega)$ to the symmetric space $ {W}^{\alpha,p}(\R)$. 
  \end{remark}
    
    \subsubsection{\bf Exterior Extensions}\label{sec-4.2.3}
     In this subsection we construct a more traditional (exterior) extension so that the 
     extended function 
     coincides with the original function in the entire domain $\Omega$ where the latter is defined. As we alluded  earlier, if we do not want to pay in part of the domain, we need to pay with a restriction on 
     the function to be extended.
     
    \begin{lemma}\label{ExtensionOutsideLemma}
        Let $\Omega=(a,b)$, $0 < \alpha <1$ and $1 \leq p < \infty$. Assume that $\alpha p<1$ and $\mu \in \R$ so that $\mu > p(1-\alpha p)^{-1}$ (hence, $\mu>p$). Then for every bounded domain 
        $\Omega ' \supset \supset \Omega$, 
        there exists a constant $C = C(\alpha,p,\Omega') >0$ such that for every $u \in {^{\pm}}{W}{^{\alpha,p}}(\Omega) \cap L^{\mu}(\Omega)$, there exists $u^{\pm} \in {^{\pm}}{W}{^{\alpha,p}}(\Omega')$ such that 
        \begin{enumerate}
            \item[{\rm (i)}] $u^{\pm} = u$ a.e. in $\Omega$,
            \item[{\rm (ii)}] $\supp(u^{\pm})\subset\subset \Omega'$,
            \item[{\rm (iii)}] $\|u^{\pm}\|_{{^{\pm}}{W}{^{\alpha,p}}(\Omega')} \leq C \|u\|_{{^{\pm}}{W}{^{\alpha,p}}(\Omega)}$. 
        \end{enumerate}
    \end{lemma}
    
    \begin{proof}
        Let $u \in {^{\pm}}{W}{^{\alpha,p}}(\Omega) \cap L^{\mu}(\Omega)$. For ease of
        presentation and without loss of the generality,  we only consider the left weak fractional  derivative  with $\Omega = (0,1)$.
        
         Let $\Omega' = (-1 ,2)$, $\{B_{i}\}_{i=1}^{N} \subset \Omega'$ be a cover of 
        $\overline{\Omega}$ and $\{\psi_i\}_{i=1}^{N}$ be a subordinate partition of unity so that 
        $\supp(\varphi_i) \subset B_i$ for $i=1,2,\cdots, N$. 
         Define $u^{-} = \overline{u}^{-} \psi$ in $\Omega'$, where 
        \begin{align*}
           \psi:=\sum_{i=1}^{N}\psi_{i}, \qquad 
            \overline{u}^{-}(x):= \begin{cases}
                0 &\text{if }  x \in (-1,0), \\
                u &\text{if } x \in (0,1) ,\\ 
                u(x-1) &\text{if } x \in (1,2).
            \end{cases}
        \end{align*}
       Notice that $\overline{u}^{-}$ is a periodic extension of $u$ to the right on interval $(1,2)$.  
        
        Trivially, $\|u^{-}\|_{L^{p}(\Omega')} \leq 2 \|u\|_{L^p(\Omega)}$. It remains to prove that $u^{-}$ is weakly differentiable in $L^{p}(\Omega')$. To the end, let $\{u_j\}_{j= 1}^{\infty} \subset C^{\infty}(\Omega)$ such that $u_j  \rightarrow u$ in ${^{-}}{W}{^{\alpha,p}}(\Omega) \cap L^{\mu}(\Omega)$ as
        $j\to \infty$. Let $u_j^-:=\overline{u}^{-}_{j} \psi$ and $\overline{u}^{-}_{j} $ is the extension of $u_j$ to $\Omega'$
        constructed in the same way  as $\overline{u}^{-}$ is done above.  
           %
         Since $u_j\to u$ in $L^\mu(\Omega)$, by the construction, we  have $\overline{u}^{-}_j \to \overline{u}^{-}$ and  
         $u_j^- \to  u^-r$ in $L^{\mu}(\Omega')$. 
         Hence, $\{u_j^-\}_
         {j=1}^{\infty}$ is bounded in $L^{\mu}(\Omega')$. $\{{^{-}}{\mathcal{D}}{^{\alpha}}u_j\}_{j=1}^{\infty}$ is bounded in $L^{p}(\Omega)$ because  ${^{-}}{\mathcal{D}}{^{\alpha}}u_j\to {^{-}}{\mathcal{D}}{^{\alpha}}u$ in $L^{p}(\Omega)$.  Let $M >0$ be such a bound for both sequences. 
         
         Now, using the fact that ${^{-}}{\mathcal{D}}{^{\alpha}} u_j^- = {^{-}}{D}{^{\alpha}} u_j^-$
         and the definition of $u_j^-$ we have 
        \begin{align*}
            \|{^{-}}{\mathcal{D}}{^{\alpha}} u_j^-\|_{L^{p}(\Omega')}^{p} &=
            \|{^{-}}{D}{^{\alpha}} u_j^-\|_{L^{p}(\Omega')}^{p} 
            = \int_{-1}^{2}\left|\dfrac{d}{dx} \int_{-1}^{x} \dfrac{u_j^-(y)}{(x-y)^{\alpha}}\right|^{p}\,dx \\ 
            &\leq \int_{0}^{1}\left|\dfrac{d}{dx} \int_{0}^{x} \dfrac{u_j(y)}{(x-y)^{\alpha}}\,dy \right|^{p}\,dx 
            + \int_{1}^{2} \left|\int_{0}^{1} \dfrac{u_j(y)}{(x-y)^{1+\alpha}}\,dy \right|^{p}\,dx \\
            &\qquad +  \int_{1}^{2}\left|\dfrac{d}{dx} \int_{1}^{x} \dfrac{\overline{u}_j\psi}{(x-y)^{\alpha}}\,dy \right|^{p}\,dx \\
            &\leq \|{^{-}}{D}{^{\alpha}} u_j \|_{L^{p}((0,1))}^{p} + \int_{1}^{2} \left|\int_{0}^{1} \dfrac{u_j(y)}{(x-y)^{1+\alpha}}\,dy \right|^{p}\,dx  \\
            &\qquad + \int_{1}^{2}\left|\dfrac{d}{dx} \int_{1}^{x} \dfrac{\overline{u}_j\psi}{(x-y)^{\alpha}}\,dy \right|^{p}\,dx.
        \end{align*}
        
        Next, we bound  the last two terms above separately.
        To bound the second to the last (middle) term,  let  $\nu$ be the H\"older conjugate of $\mu$, then we have 
        \begin{align*}
            \int_{1}^{2} \biggl|\int_{0}^{1} \dfrac{u_j(y)}{(x-y)^{1+\alpha}}\,dy &\biggr|^{p}\,dx \leq \|u_j\|_{L^{\mu}((0,1))}^{p} \int_{1}^{2}\left(\int_{0}^{1}\dfrac{dy}{(x-y)^{\nu(1+\alpha)}}\right)^{\frac{p}{\nu}}\,dx\\
            &\leq M^{p}\| u_j \|_{L^{\mu}((0,1))}^{p} \int_{1}^{2} x^{\frac{p}{\nu} - p(1+\alpha)} + (x-1)^{\frac{p}{\nu} -p(1+\alpha)}\,dx.
        \end{align*}
        In order for this term to be finite, $p\nu^{-1}- p(1+\alpha)>-1 $ must holds,  which 
        implies that $\mu > p(1 - \alpha p)^{-1}$, which is assumed in the statement of the theorem. 
        
        Lastly,  to bound the final term, using the product rule we get 
        \begin{align*}
            \int_{1}^{2} \left|\dfrac{d}{dx} \int_{1}^{x} \dfrac{\overline{u}_{j} (y) \psi(y)}{(x-y)^{\alpha}}\,dy \right|^{p}\,dx 
            &\leq \|{^{-}}{D}{^{\alpha}} u_j \|_{L^{p}((0,1))}^{p} + \Bigl\|\sum_{k=1}^{m} C_{k} {^{-}}{I}{^{k-\alpha}}  u_j\, D^k   \psi \Bigr\|_{L^{p}((1,2))}^{p} \\
            &\qquad + \|{^{-}}{R}{^{\alpha}_{m}}(u_j,  \psi)\|_{L^{p}((1,2))}^{p} \\
            &\leq C \Bigl(\|{^{-}}{\mathcal{D}}{^{\alpha}} u_j\|_{L^{p}((0,1))}^{p} + \|u_j\|_{L^p((0,1))}^{p} \Bigr).
        \end{align*}
        It follows for given $\eps > 0$ and sufficiently large $m,n$, 
        \begin{align}
            \|{^{\pm}}{\mathcal{D}}{^{\alpha}} u_m^{\pm} - {^{\pm}}{\mathcal{D}}{^{\alpha}} u_n^{\pm} \|_{L^{p}(\Omega')}^{p} \leq C\left( \|u_m - u_n \|_{{^{\pm}}{W}{^{\alpha,p}}(\Omega)}^{p} + \| u_m - u_n \|_{L^{\mu}(\Omega)}^{p}\right) < \eps.
        \end{align}
        Therefore, there exists $v \in L^{p}(\Omega')$ so that ${^{\pm}}{\mathcal{D}}{^{\alpha}}u_j^{\pm} \rightarrow v$ in $L^{p}(\Omega')$. It is easy to see then that $v = {^{\pm}}{\mathcal{D}}{^{\alpha}} u^{\pm}$ using the definition of the weak derivative. Hence ${^{\pm}}{\mathcal{D}}{^{\alpha}}u^{\pm} \in L^{p}(\Omega')$. This completes the proof. 
    \end{proof}

\begin{remark}
(a) We note that there is no redundancy in assumption 
$u \in {^{\pm}}{W}{^{\alpha,p}}(\Omega) \cap L^\mu(\Omega)$ for $\mu > p(1-\alpha p)^{-1}$ 
because it will be proved in Section \ref{sec-4.3} that ${^{\pm}}{W}{^{\alpha,p}}(\Omega)$ 	
is not embedded into $L^{\mu}(\Omega)$ in general. 

(b)  It can be proved that the condition $\alpha p < 1$ and $u \in L^{\mu}(\Omega)$ for some $p < \mu \leq \infty$ 
	are necessary (given the current  calculations). 
	In order for the kernel to remain bounded, we must impose the condition 
	$-1 < p{\nu}^{-1} -p(1+\alpha) < 0$ which implies that  $(1-\alpha p)  > p\mu^{-1}$. 
	Thus,  it follows from $(1- \alpha p) > 0$ that $\alpha p < 1$.
	This shows that $\alpha p <1$ is a necessary condition for the integrability of the kernel function 
	using an estimate as shown above. Moreover, if $\mu = p$, then $\nu = p(p-1)^{-1}$ and the inequality 
	$-1 < p{\nu}^{-1} -p(1+\alpha)$ implies that $\alpha p<0$, which is a contradiction. Hence, 
	we must take $\mu > p$. In particular, $\mu = \infty$ is allowed though not necessary. We need only assume 
	that $u \in L^{\mu}(\Omega)$ with the condition $\mu > p/(1-\alpha p)$.
	
(c) The same result can be proven for $u \in{W}^{\alpha,p}(\Omega)\cap L^{\mu}(\Omega)$. In this case, $\overline{u} : = \overline{u}^{\pm}$ is taken to be the periodic extension over all of $\Omega'$.
\end{remark}
 
    \begin{theorem}\label{ExteriorExtension}
         Let $\Omega=(a,b)$, $0 < \alpha <1$ and $1 \leq p < \infty$. Assume that $\alpha p<1$ and $\mu \in \R$ so that $\mu > p(1-\alpha p)^{-1}$ (hence, $\mu>p$). Then for every bounded domain 
        $\Omega ' \supset \supset \Omega$,      
        there exists mappings $E_{\pm} :{^{\pm}}{W}{^{\alpha,p}}(\Omega)\cap L^{\mu}(\Omega) \rightarrow {^{\pm}}{W}{^{\alpha,p}}(\R)$ and $C = C(\alpha,p,\Omega') > 0$ such that for any $u \in {^{\pm}}{W}{^{\alpha,p}}(\Omega) \cap L^{\mu}(\Omega)$, 
        \begin{enumerate}
            \item[{\rm (i)}] $E_{\pm}u = u$ a.e in $\Omega$, 
            \item[{\rm (ii)}] $\supp(E_{\pm}u)\subset\subset \Omega'$,
            \item[{\rm (iii)}] $\|E_{\pm}u\|_{{^{\pm}}{W}{^{\alpha,p}}(\R)} \leq C
            \bigl(\|u\|_{{^{\pm}}{W}{^{\alpha,p}}(\Omega)} + \|u\|_{L^{\mu}(\Omega)} \bigr).$
        \end{enumerate}
    \end{theorem}

    \begin{proof}
    	For any $u \in {^{\pm}}{W}{^{\alpha,p}}(\Omega) \cap L^{\mu}(\Omega)$, let 
    	$u^{\pm} \in {^{\pm}}{W}{^{\alpha,p}}(\Omega')$ be the function defined in 
        Lemma \ref{ExtensionOutsideLemma} and set $E_{\pm}u = \widetilde{u^{\pm}}$, the trivial extension of $u^{\pm}$.
        It follows from  Lemma \ref{TrivialExtension}  that $E_{\pm}$ satisfies the desired properties.
        The proof is complete. 
    \end{proof}

\begin{corollary}
The conclusion of Theorem \ref{ExteriorExtension} also holds for functions in ${W}^{\alpha,p}(\Omega) \cap L^{\mu}(\Omega)$.
\end{corollary}

    %
   
    
    \subsection{One-side Boundary Traces and Compact Embedding}\label{sec-4.3}
    Similar to the integer order case, since functions in Sobolev spaces ${^{\pm}}{W}{^{\alpha,p}}((a,b))$
    are integrable functions, a natural question is under what condition(s) those functions can be 
    assigned pointwise values, especially, at two boundary points $x=a,b$. Such a question arises naturally when 
    studying initial and initial-boundary value problems for fractional order differential equations. 
    It turns out that the situation is more delicate in the fractional order case because the existence of the kernel functions creates a hick-up in this pursuit. We shall establish a one-side embedding result for each of spaces 
    ${^{\pm}}{W}{^{\alpha,p}}((a,b))$, which then allows us to assign one-side traces for those functions.  
    First, we establish the following classical characterization of Sobolev functions.
    
    \begin{proposition}\label{ContinuousRepresentative}
        Let $(a,b)\subset \R$, $0< \alpha <1$, $1 \leq p \leq \infty$ so that $\alpha p > 1$.
        \begin{itemize}
            \item[(i)] If $u \in {^{-}}{W}{^{\alpha,p}}((a,b))$, then for any $c \in (a,b)$, there exists $\bar{u} \in C([c,b])$ so that $u = \bar{u}$ almost everywhere in $[c,b]$. 
            \item[(ii)] If $u \in {^{+}}{W}{^{\alpha,p}}((a,b))$, then for any $c \in (a,b)$, there exists $\bar{u} \in C([a,c])$ so that $u = \bar{u}$ almost everywhere in $[a,c]$.
            \item[(iii)] If $u \in W^{\alpha,p}((a,b))$, then there exists $\bar{u} \in C([a,b])$ so that $u = \bar{u}$ almost everywhere in $[a,b]$. 
        \end{itemize}
    \end{proposition}
    
    \begin{proof}
        We only give a proof for $(i)$ because $(ii)$ follows similarly and $(iii)$ is proved by combining $(i)$ and $(ii)$. Let $u \in {^{-}}{W}{^{\alpha,p}}((a,b))$ and set $u^* = {^{-}}{I}{^{\alpha}} {^{-}}{\mathcal{D}}{^{\alpha}} u$. Then for any $\varphi \in C^{\infty}_{0}((a,b))$, it follows by the $L^{p}$ mapping properties of fractional integrals, classical fractional integration by parts, and the definition of weak fractional derivatives (cf. \cite[Theorem 2.5 and 2.6]{Feng_Sutton}) that  
        \begin{align*}
            \int_{a}^{b} u^* {^{+}}{D}{^{\alpha}} \varphi\, dx &= \int_{a}^{b} \varphi  {^{-}}{D}{^{\alpha}} u^* \, dx
            = \int_{a}^{b} \varphi  {^{-}}{D}{^{\alpha}} {^{-}}{I}{^{\alpha}} {^{-}}{\mathcal{D}}{^{\alpha}} u  \, dx \\
            &= \int_{a}^{b}  \varphi {^{-}}{\mathcal{D}}{^{\alpha}} u \, dx
            = \int_{a}^{b} u {^{+}}{D}{^{\alpha}} \varphi \, dx.
        \end{align*}
        Consequently,
        \begin{align*}
            0 = \int_{a}^{b} (u - u^*) {^{+}}{D}{^{\alpha}} \varphi\, dx = \int_{a}^{b} {^{-}}{I}{^{1-\alpha}}(u - u^*) \varphi'\, dx.
        \end{align*}
        Thus, ${^{-}}{I}{^{1-\alpha}}u - {^{-}}{I}{^{1-\alpha}} u^* = C$ a.e. in $(a,b)$. It follows  from the Fundamental Theorem of Classical Fractional Calculus (FTcFC, cf. \cite[Lemma 3.1]{Feng_Sutton}) that 
        $u = u^* + {^{-}}{D}{^{1-\alpha}} C$ almost everywhere. Choose $\bar{u} = u^* + {^{-}}{D}{^{1-\alpha}} C$, we have that $\bar{u} \in C([c,b])$ for every $c \in(a,b)$ and $u = \bar{u}$ almost everywhere. 
    \end{proof}
    
    \begin{remark}
        If a function $u$ belongs to ${^{\pm}}{W}{^{\alpha,p}}$, then any function $v = u$ almost everywhere must also belong to ${^{\pm}}{W}{^{\alpha,p}}$. Therefore, we do not differentiate between any two functions that may only differ from one another on a measure zero set. Proposition \ref{ContinuousRepresentative} asserts that every function $u \in {^{-}}{W}{^{\alpha,p}}((a,b))$ admits a continuous representative on $[c,b]$. When it is helpful, (i.e. giving meaning to $u(x)$ for some $x \in [c,b]$) we replace $u$ with its continuous representative $\bar{u}$. In order to avoid confusion and eliminate unnecessary notation, we will still use $u$ to denote the continuous representative.
    \end{remark}

     \begin{theorem}\label{CompactEmbedding}
    	Let $(a,b) \subset \R$, $0 < \alpha <1$ and $1 < p < \infty$. Suppose that $\alpha p >1$. 
    	\begin{itemize}
    		\item[{\rm (i)}] If $u \in {^{-}}{W}{^{\alpha,p}}((a,b))$, then for any $c \in (a,b)$, 
    		the injection ${^{-}}{W}{^{\alpha,p}}((a,b)) \hookrightarrow C^{\alpha -\frac{1}{p}}([c,b])$
    		is compact. 
     
    		\item[{\rm (ii)}] If $u \in {^{+}}{W}{^{\alpha,p}}((a,b))$, then for any $c\in (a,b)$, the injection ${^{+}}{W}{^{\alpha,p}}((a,b)) \hookrightarrow C^{\alpha -\frac{1}{p}}([a,c])$
    		is compact. 
    		
    		\item[{\rm (iii)}] If $u \in W^{\alpha,p}((a,b))$, then the injection $W^{\alpha,p}((a,b)) \hookrightarrow C^{\alpha -\frac{1}{p}}([a,b])$ is compact. 
    	\end{itemize}
        
    \end{theorem}
    
    \begin{proof}
         We only give a proof for (i) because the other two cases follow similarly. 
       
       Let ${B}_1^-$ be the unit ball in ${^{-}}{W}{^{\alpha,p}}((a,b))$ and take $u \in {B}_1^-$. Let $c \in (a,b)$. For any two distinct points $x,y \in [c,b]$ (assume $x>y$), by the FTwFC (cf. Theorem \ref{FTWFC}) we get 
        \begin{align}\label{a1}
             \bigl|u (x) - u (y)\bigr| 
            &= \biggl|c_{-}^{1-\alpha} [(x-a)^{\alpha -1} - (y-a)^{\alpha -1}] + C_{\alpha} \int_{y}^{x} \dfrac{{^{-}}{\mathcal{D}}{^{\alpha}}u(z)}{(x-z)^{1-\alpha}}\,dz \\
            &\qquad  + C_{\alpha} \int_{a}^{y} \dfrac{{^{-}}{\mathcal{D}}{^{\alpha}}u(z)}{(x-z)^{1-\alpha}} - \dfrac{{^{-}}{\mathcal{D}}{^{\alpha}} u(z)}{(y-z)^{1-\alpha}}\,dz \biggr| \nonumber \\
            &\leq \left|c_{-}^{1-\alpha} [(x-a)^{\alpha -1} - (y-a)^{\alpha -1}]\right| + C_{\alpha} \left|\int_{y}^{x} \dfrac{{^{-}}{\mathcal{D}}{^{\alpha}}u(z)}{(x-z)^{1-\alpha}}\,dz \right| \nonumber \\
            &\quad  + C_{\alpha} \left| \int_{a}^{y} \dfrac{{^{-}}{\mathcal{D}}{^{\alpha}} u(z) [ (y-z)^{1-\alpha} - (x-z)^{1-\alpha}]}{[(y-z)(x-z)]^{1 - \alpha}}\,dz\right|. \nonumber
            \end{align}
            
           Below we bound each of the three terms on the right-hand side.
           Upon noticing that $|c_{-}^{1-\alpha}| \leq C_{\Omega,\alpha ,p} \|u\|_{{^{-}}{W}{^{\alpha,p}}(\Omega)}$, 
            \begin{align}\label{a2}
            \left|c_{-}^{1-\alpha} [ (x-a)^{\alpha -1} - (y-a)^{\alpha -1} ]\right| 
            &\leq C_{\Omega , \alpha , p} \|u\|_{{^{-}}{W}{^{\alpha,p}}(\Omega)} |\xi - a|^{\alpha -1} |x-y|\\
            &\leq C_{\Omega,\alpha,p} \|u\|_{{^{-}}{W}{^{\alpha,p}}(\Omega)} |x-y|^{\alpha - \frac{1}{p}}, \nonumber
            \end{align}
            \begin{align}\label{a3}
            C_{\alpha} \left|\int_{y}^{x} \dfrac{{^{-}}{\mathcal{D}}{^{\alpha}}u(z)}{(x-z)^{1-\alpha}}\,dz \right|
            &\leq C_{\alpha} \left\| {^{-}}{\mathcal{D}}{^{\alpha}} u \right\|_{L^{p}((a,b))} \biggl| \int_{y}^{x} (x-z)^{-q(1-\alpha)}\,dz \biggl|^{\frac{1}{q}} \\
            &\leq C_{\alpha , p} \left\|{^{-}}{\mathcal{D}}{^{\alpha}} u \right\|_{L^{p}((a,b))} |x-y|^{\alpha-\frac{1}{p}}, \nonumber
            \end{align}
            \begin{align}\label{a4}
            & C_{\alpha}  \int_{a}^{y} \dfrac{\left|{^{-}}{\mathcal{D}}{^{\alpha}} u(z)\right| \left|(y-z)^{1-\alpha} - (x-z)^{1-\alpha}\right|}{\left|(y-z)(x-z)\right|^{1 - \alpha}}\,dz\\
            &\hskip 1.2in \leq C_{\alpha,p} \left\| {^{-}}{\mathcal{D}}{^{\alpha}} u \right\|_{L^{p}((a,b))} |x-y|^{\alpha- \frac{1}{p}}. \nonumber
            \end{align}
            %
        
        Substituting \eqref{a2}--\eqref{a4} into \eqref{a1} yields 
        \begin{align}\label{a5}
        |u(x) - u(y)| \leq C|x-y|^{\alpha -\frac{1}{p}} \qquad \forall x, y\in [c,b],  
        \end{align}
        where $C$ is a positive constant independent of $x$ and $y$. 
        Because $\alpha -\frac{1}{p}>0$, then ${B}^{-}_{1}$ is uniformly equicontinuous in $C([c,b])$. It follows from  Arzel\`a-Ascoli theorem that ${B}^{-}_{1}$ has compact closure in $C^{\alpha - \frac{1}{p}}([c,b])$. The proof is complete. 
    \end{proof}

    \begin{remark}
    	(a) We note that unlike the integer order case, we have proved the above 
    	embedding results directly rather than relying on the infinite domain results and extension theorem. 
    	
    	(b) From the above calculations we observe that when $c_{\pm}^{1-\alpha} = 0$, the injection can be extended to the initial boundary so that ${^{\pm}}{W}{^{\alpha,p}}((a,b)) \hookrightarrow C^{\alpha - \frac{1}{p}}([a,b])$. Effectively,  $c^{1-\alpha}_{\pm} = 0$ implies that any singularity at the initial boundary is prevented; we denote this space by 
        \begin{align}\label{mathring}
        {^{\pm}}{\mathring{W}}{^{\alpha,p}}(\Omega) : = \{ u \in {^{\pm}}{W}{^{\alpha,p}}(\Omega) \, : \, c^{1-\alpha}_{\pm} = 0\}.
        \end{align}
    \end{remark}
    
The above embedding theorem motivates us to introduce the following definition of trace operators. 

    \begin{definition}
        Define trace operators ${^{-}}{T}: {^{-}}{W}{^{\alpha,p}}((a,b))\to \R$ by ${^{-}}{T}u={^{-}}{T}u|_{x=b} := u(b)$ and ${^{+}}{T}: {^{+}}{W}{^{\alpha,p}}((a,b))\to \R$ by ${^{+}}{T}u={^{+}}{T}u|_{x=a} := u(a)$.  
    \end{definition}
    
    It should be noted that the above proof demonstrates that we can confirm the following trace inequality: 
    \begin{align}\label{TraceInequality}
    	|{^{\pm}}{T} u | \leq C \|u\|_{{^{\pm}}{W}{^{\alpha,p}}(\Omega)}.
    \end{align}
    
    \subsubsection{\bf Zero Trace Spaces}\label{sec-4.3.1}
    With the help of the trace operators in spaces ${^{\pm}}{W}{^{\alpha,p}}(\Omega)$, 
    we can define and characterize different spaces with zero trace. First, we explicitly 
    define the zero trace spaces  and a new norm for these spaces.

    \begin{definition}
        Let $\Omega=(a,b)$, $0 < \alpha <1$ and $1 < p <\infty$. Suppose that  $\alpha p >1$.  Define  
        \begin{align*}
            {^{\pm}}{W}{^{\alpha,p}_{0}}(\Omega) &:= \{ u \in {^{\pm}}{W}{^{\alpha,p}}(\Omega) \,:\, {^{\pm}}{T} u = 0 \},\\
             {W}^{\alpha,p}_{0}(\Omega) &:= \{ u \in  {W}^{\alpha,p}(\Omega) : {^{-}}{T}u=0 \mbox{ and } {^{+}}{T}u = 0\};
        \end{align*}
    and the norm $\| u \|_{\leftidx{^{\pm}}{W}{^{\alpha,p}_{0}}(\Omega)} : = \|\leftidx{^{\pm}}{\mathcal{D}}{^{\alpha}} u \|_{L^{p}(\Omega)}$ for $1< p  <\infty$.
    \end{definition}
    
    \begin{proposition}\label{ZeroTraceNorm}
    $\| u \|_{\leftidx{^{\pm}}{W}{^{\alpha,p}_{0}}(\Omega)}$ defines a norm.
    \end{proposition}

We omit the proof because it is straightforward. 
    
    
    In an effort to characterize the above spaces, our goal is to link these spaces with the completion spaces introduced in Section \ref{sec-4.1}. As our notion of traces is one-sided,  
    this makes the use of one-sided approximations spaces (i.e. ${^{\pm}}{C}{^{\infty}_{0}}(\Omega)$) 
    sensible.

    \begin{lemma}
         Let $\Omega=(a,b)$, $0 < \alpha <1$ and $1 < p <\infty$. Suppose that  $\alpha p >1$. If $u \in {^{\pm}}{W}{^{\alpha,p}}(\Omega) \cap {^{\pm}}{C}{^{\infty}_{0}}(\Omega)$, then $u \in {^{\pm}}{\overline{W}}{^{\alpha,p}_{0}}(\Omega)$.
    \end{lemma}

    \begin{proof}
        Let $u \in {^{\pm}}{W}{^{\alpha, p}}(\Omega) \cap {^{\pm}}{C}{^{\infty}_{0}}(\Omega)$. 
        Consider the sequence $u_j := \eta_{\frac{1}{j}}*u$ with $\eta$ being the standard mollifier. Then $u_j \in {^{\pm}}{W}{^{\alpha,p}}(\Omega) \cap {^{\pm}}{C}{^{\infty}_{0}}(\Omega)$ and $u_j \rightarrow u$ in ${^{\pm}}{W}{^{\alpha,p}}(\Omega)$. Thus $u \in {^{\pm}}{\overline{W}}{^{\alpha,p}_{0}}(\Omega)$. 
    \end{proof}

The next two theorems give characterizations of the zero trace spaces.

    \begin{theorem}\label{ZeroTrace}
         Let $\Omega=(a,b)$, $0 < \alpha <1$ and $1 < p <\infty$. Suppose that  $\alpha p >1$.
         Then $ {^{\pm}}{\overline{W}}{^{\alpha , p}_{0}}(\Omega)= {^{\pm}}{W}{^{\alpha,p}_{0}}(\Omega)$ and ${\overline{W}}{^{\alpha,p}_{0}}(\Omega) = W^{\alpha,p}_{0}(\Omega)$.
    \end{theorem}

    \begin{proof}
        Let $u \in {^{\pm}}{\overline{W}}{^{\alpha,p}_{0}}(\Omega)$. Then there exists $\{u_j\}_{j=1}^{\infty} \subset {^{\pm}}{C}{^{\infty}_{0}}(\Omega)$ so that $u_j \rightarrow u$ in ${^{\pm}}{W}{^{\alpha,p}}(\Omega)$. 
        It follows that ${^{\pm}}{T} u_j = 0$ and $u_j \rightarrow u$ uniformly on $[c,b]$ or $[a,c]$ for every
        $c\in (a,b)$. Consequently, ${^{\pm}}{T} u = 0$. Thus, $ {^{\pm}}{\overline{W}}{^{\alpha , p}_{0}}(\Omega)\subset {^{\pm}}{W}{^{\alpha,p}_{0}}(\Omega)$

        Conversely, let $u \in {^{\pm}}{W}{^{\alpha,p}_{0}}(\Omega)$. We want to show that there exists $\{u^n\} \subset {^{\pm}}{C}{^{\infty}_{0}}(\Omega)$ such that $u^n \rightarrow u$ in ${^{\pm}}{W}{^{\alpha,p}}(\Omega)$. 
        For ease of presentation and without loss of the generality, let $\Omega = (0,1)$ and we only consider the left 
        space. Fix a function $\varphi \in C^{\infty}(\R)$ such that 
        \begin{align*}
            \varphi(x) : = \begin{cases} 0 &\text{if } |x| \leq 1, \\ 
            x &\text{if } |x| \geq 2,
            \end{cases}
        \end{align*}
        and $|\varphi(x)| \leq |x|$. Choose $\{u_j\}_{j=1}^{\infty} \subset C^{\infty}(\Omega)$ so that $u_j \rightarrow u$ in ${^{-}}{W}{^{\alpha,p}}(\Omega)$ and define the sequence $u^{n}_{j} : = (1/n) \varphi(nu_j)$. We can show that $u^{n}_{j} \rightarrow u^n$ in $L^{p}((0,1))$. Moreover, using the chain rule (cf. \cite[Theorem 4.4]{Feng_Sutton}),
        we get
        \begin{align*}
            &\|{^{-}}{\mathcal{D}}{^{\alpha}} u^{n}_{j}\|_{L^{p}((0,1))}
            = \dfrac{1}{n} \left\| \dfrac{\varphi(nu_j)}{nu_j} {^{-}}{\mathcal{D}}{^{\alpha}} nu_j + {^{-}}{R}{^{\alpha}_{0}} \left(nu_j , \dfrac{\varphi(nu_j)}{nu_j}\right) \right\|_{L^{p}((0,1))}\\
            &\qquad \leq \dfrac{1}{n} \left\|\dfrac{\varphi(nu_j)}{nu_j} {^{-}}{\mathcal{D}}{^{\alpha}} nu_j \right\|_{L^{p}((0,1))}^{p} + \dfrac{1}{n} \left\| {^{-}}{R}{^{\alpha}_{0}} \left(nu_j , \dfrac{\varphi(nu_j)}{nu_j}\right) \right\|_{L^{p}((0,1))}\\
            &\qquad \leq \|{^{-}}{\mathcal{D}}{^{\alpha}} u_j \|_{L^{p}((0,1))}^{p} + \dfrac{1}{n} \left\| {^{-}}{R}{^{\alpha}_{0}} \left(nu_j , \dfrac{\varphi(nu_j)}{nu_j}\right) \right\|_{L^{p}((0,1))},
        \end{align*}
        where 
        \begin{align*}
            &\left\| {^{-}}{R}{^{\alpha}_{0}} \left(nu_j , \dfrac{\varphi(nu_j)}{nu_j}\right) \right\|_{L^{p}((0,1))}^{p} \\
             &\quad =  \int_{0}^{1} \left| \int_{0}^{x} \dfrac{nu_j(y)}{(x-y)^{1+\alpha}} \left( \dfrac{\varphi(nu_j)(x)}{nu_j(x)} - \dfrac{\varphi(nu_j)(y)}{nu_j(y)} \right)\,dy \right|^{p}\,dx\\
            &\quad \leq  \int_{0}^{1} \left( \int_{0}^{x} \dfrac{|u_j(y)|}{(x-y)^{1+\alpha}} \dfrac{2n|u_j(x) - u_j(y)|}{|u_j(y)|}\,dy \right)^{p}\,dx \\
            &\quad \leq 2^p n^p \int_{0}^{1} \left( \int_{0}^{x} \dfrac{dy}{(x-y)^{\alpha}} \right)^{p}\,dx \leq 2^{p} n^{p}. 
        \end{align*}
        Hence, $\{u^{n}_{j}\}_{j=1}^{\infty}$ is a bounded sequence in ${^{-}}{W}{^{\alpha,p}}((0,1))$. Thus there exists $v^n \in L^{p}((0,1))$ so that ${^{-}}{\mathcal{D}}{^{\alpha}} u^n_j \rightharpoonup v^n$ in $L^{p}((0,1))$ as 
        $j\to \infty$. It can easily be shown using the weak derivative definition that $v^{n} = {^{-}}{\mathcal{D}}{^{\alpha}}u^{n}$. Hence $\{u^n\}_{n=1}^{\infty}$ belongs to ${^{-}}{W}{^{\alpha,p}}((0,1))$. On the other hand, since ${^{-}}{T}u = 0 $, 
        then $u_n \in {^{-}}{C}{^{\infty}_{0}}((0,1))$. Finally, it is a consequence of Lebesgue Dominated Convergence theorem that $u^n \rightarrow u$ in ${^{-}}{W}{^{\alpha,p}}((0,1))$. Thus, $u \in {^{-}}{\overline{W}}{^{\alpha,p}_{0}}((0,1))$. The proof is complete. 
    \end{proof}

    \begin{theorem}
          Let $\Omega=(a,b)$, $0 < \alpha <1$ and $1 < p <\infty$. Suppose that  $\alpha p >1$. Then $\overline{W}^{\alpha ,p}_{0}(\Omega)= W^{\alpha,p}_{0}(\Omega)$. 
    \end{theorem}
    
    \begin{proof}
    The same construction and proof used for the one-sided closure spaces in Theorem \ref{ZeroTrace} can be used for the symmetric result $W^{\alpha,p}_{0} = \overline{W}^{\alpha,p}_{0}$.
    \end{proof}
    
    At this point, we have gathered sufficient tools to prove a crucial characterization 
    result and a pair of integration by parts 
    formula for functions in the symmetric fractional Sobolev spaces $W^{\alpha,p}(\Omega)$. Similar integration by parts formula in ${^{\pm}}{W}{^{\alpha,p}}(\Omega)$ will be presented in a subsequent section. 
    
\begin{proposition}\label{ConstantZero}
Let $\Ome=(a,b)$. If $u \in \leftidx{}{W}{^{\alpha,p}_{0}}(\Omega)$, then $\leftidx{^{+}}{T}\leftidx{^{-}}{I}{^{\alpha}} u = \leftidx{^{-}}{T} \leftidx{^{+}}{I}{^{\alpha}} u = 0$. That is, $c^{1-\alpha}_{+} = c^{1-\alpha}_{-} = 0.$
\end{proposition}

\begin{proof}
    Let $u \in W^{\alpha,p}_{0}(\Omega)$. It follows that $u \in C(\overline{\Omega})$.  
    Then we have 
    \begin{align*}
        \lim_{x \rightarrow a} \left|\int_{a}^{x} \dfrac{u(y)}{(x-y)^{1- \alpha}} \,dy \right| &= \lim_{x \rightarrow a} \left| \int_{0}^{x-a} \dfrac{u(x-z)}{z^{1-\alpha}}\,dz \right|\\
        &\leq \lim_{x \rightarrow a} \|u\|_{L^{\infty}(\Omega)} \int_{0}^{x-a} \dfrac{dz}{z^{1-\alpha}} \\ 
        &= \lim_{x\rightarrow a}\|u\|_{L^{\infty}(\Omega)} (x-a)^{\alpha} 
        = 0.
    \end{align*}
    The other trace follows similarly. 
\end{proof}

    \begin{proposition}
        Let $\Omega \subset \R$, $\alpha >0$ and $1 \leq p, q < \infty$. Suppose that $\alpha p> 1$ and $\alpha q> 1$. Then for any $u \in W^{\alpha,p}(\Omega)$ and $v \in W^{\alpha ,q}(\Omega)$, there holds the following integration by parts formula: 
        \begin{align}\label{SymmetricSobolevIBP}
            \int_{\Omega} u {^{\pm}}{\mathcal{D}}{^{\alpha}} v\, dx = (-1)^{[\alpha]} \int_{\Omega} v {^{\mp}}{\mathcal{D}}{^{\alpha}} u\, dx.
        \end{align}
    \end{proposition}
    
    \begin{proof}
        We only give a proof for $0 < \alpha <1$ because the other cases follow similarly. By Theorem \ref{H=W} and Theorem \ref{CompactEmbedding}, there exist  $\{u_j\}_{j=1}^{\infty} \subset C^{\infty}(\Omega) \cap C(\overline{\Omega})$ and $\{v_k\}_{k=1}^{\infty}\subset C^{\infty}(\Omega) \cap C(\overline{\Omega})$ such that $u_j\to u$ in $W^{\alpha,p}(\Omega)$ and $v_k\to v$
        in $W^{\alpha,q}(\Omega)$. It follows by the classical integration by parts that 
        \begin{align*}
            \int_{\Omega} u {^{\mp}}{\mathcal{D}}{^{\alpha}} v \,dx= \lim_{j,k \rightarrow \infty} \int_{\Omega} u_j {^{\mp}}{D}{^{\alpha}} v_k\, dx = \lim_{j,k \rightarrow \infty} \int_{\Omega} v_k {^{\pm}}{D}{^{\alpha}} u_j \, dx
            = \int_{\Omega} v {^{\pm}}{\mathcal{D}}{^{\alpha}} u\, dx. 
        \end{align*}
        This completes the proof.
    \end{proof}
    
    \begin{remark}
        We have used the fact that $u$ and $v$ continue to the boundary of $\Omega$ in order to apply the classical integration by parts formula. Due to the inability to guarantee this for functions 
        in the one-sided spaces ${^{\pm}}{W}{^{\alpha,p}}(\Omega)$, we postpone 
        presenting a similar result in those spaces to Section \ref{sec-4.6.1}. 
    \end{remark}

    
    \subsection{Sobolev and Poincar\'e Inequalities}\label{sec-4.4}
   %
   The goal of this subsection is to extend the well known Sobolev and Poincar\'e  inequalities  
   for functions in $W^{1,p}(\Omega)$ to the fractional Sobolev spaces 
   ${^{\pm}}{W}{^{\alpha,p}}(\Omega)$. We shall present the extensions separately for  
   the infinite domain $\Omega=\R$ and the finite domain $\Omega=(a,b)$ because the kernel functions 
   have a very different boundary behavior in the two cases, which in turn results in different inequalities 
   in these two cases. 
   Two tools that will play a crucial role in our analysis are the $L^{p}$ mapping properties 
   of the fractional integral operators (cf. \cite[Theorem 2.6]{Feng_Sutton}),    
   and the  FTwFC (cf. Theorem \ref{FTWFC}). 
 
   \subsubsection{\bf The Infinite Domain Case: $\Omega=\R$}\label{sec-4.4.1}
   Due to the flexibility of the choice of $0<\alpha <1$, the validity of a Sobolev 
   inequality in the fractional order
   case has more variations depending on the range of $p$. Precisely, we have

   \begin{theorem}\label{thm_Sobolev_inq}
   	Let $0<\alpha<1$ and $1 <  p < \frac{1}{\alpha}$. Then there exists a constant $C >0$ such that for any $u\in L^1(\R) \cap {^{\pm}}{W}{^{\alpha,p}}(\R)$.
   	\begin{align}\label{SobolevInequalityR}
   	\|u\|_{L^{p^*}(\R)} \leq C \|{^{\pm}}{\mathcal{D}}{^{\alpha}} u\|_{L^{p}(\R)}, \qquad p^* : = \frac{p}{1-\alpha p}.
   	\end{align}
   	$p^*$ is called the fractional Sobolev conjugate of $p$. 
   \end{theorem}
   
   \begin{proof}
   	It follows from the density/approximation theorem that there exists a sequence $\{ u_j \}_{j =  1}^{\infty} \subset C^{\infty}_{0}(\R)$ so that $u_j\rightarrow u$ in ${^{\pm}}{W}{^{\alpha}}(\R)$. Note that by 
   	construction, we also have $u_j \rightarrow u$ in $L^{1}(\R)$. Then by the 
   	FTcFC (cf. \cite[Theorem 3.2]{Feng_Sutton}), we get
   	\begin{align*}
   	u_j(x) = {^{\pm}}{I}{^{\alpha}} {^{\pm}}{D}{^{\alpha}} u_j(x)\qquad \forall x \in \R.
   	\end{align*}
   	By the $L^{p}$ mappings properties of fractional integrals 
   	(cf. \cite[Theorem 2.6]{Feng_Sutton}), we have
   	\begin{align*}
   	 \|u_j\|_{L^{p*}(\R)} &=  \|{^{\pm}}{I}{^{\alpha}} {^{\pm}}{D}{^{\alpha}} u_j \|_{L^{p*}(\R)}
   	\leq  C \|{^{\pm}}{D}{^{\alpha}} u_j \|_{L^{p}(\R)}
   	< \infty.
   	\end{align*}
   	Consequently,
   	\begin{align*}
   	\|u_m - u_n \|_{L^{p*}(\R)} \leq C \| {^{\pm}}{\mathcal{D}}{^{\alpha}} u_m - {^{\pm}}{\mathcal{D}}{^{\alpha}} u_n \|_{L^{p}(\R)} \to 0 \quad\mbox{as } m,n\to \infty.
   	\end{align*}
   	Hence, $\{u_j\}_{j=1}^{\infty}$ is a Cauchy sequence in $L^{p*}(\R)$. Therefore, there exists 
   	a function  $v \in L^{p^*}(\R)$ so that $u_j \rightarrow v$ in $L^{p^*}(\R)$.
   	Recall that we also have $u_j \rightarrow u$ in $L^{p}(\R)$. Moreover, 
   	for every $\varphi \in C^{\infty}_{0}(\R)$ 
   	\begin{align*}
   	\int_{\R} v {^{\mp}}{D}{^{\alpha}}\varphi\,dx &= \lim_{j \rightarrow \infty} \int_{\R} u_j {^{\mp}}{D}{^{\alpha}} \varphi \, dx
   	= \lim_{j \rightarrow \infty} \int_{\R} {^{\pm}}{\mathcal{D}}{^{\alpha}} u_j \varphi \,dx \\
   	&= \int_{\R} {^{\pm}}{\mathcal{D}}{^{\alpha}} u \varphi \, dx
   	= \int_{\R} u {^{\mp}}{D}{^{\alpha}} \varphi\, dx.
   	\end{align*}
   	Thus, $v=u$ almost everywhere and 
   	\begin{align*}
   	\|u \|_{L^{p^*}(\R)} \leq C \| {^{\pm}}{\mathcal{D}}{^{\alpha}} u\|_{L^p(\R)}.
   	\end{align*}
   	The proof is complete. 
   \end{proof}
   
   \begin{remark}
   	By the simple scaling argument, which considers the scaled function $u_{\lambda}(x) := u(\lambda x)$ for $\lambda >0$,
   	it is easy to verify that $\alpha p<1$ is a necessary condition for the inequality to hold in general. 
   	Similarly, the Poincar\'e inequality  does not hold in general, as in the integer order case, when $\Omega=\R$. 
   \end{remark}

   \subsubsection{\bf The Finite Domain Case: $\Omega=(a,b)$}\label{sec-4.4.2}
   One key difference between the infinite domain case and the finite domain case is 
   that the domain-dependent kernel functions $\kappa^{\alpha}_{-}(x):=(x-a)^{\alpha-1}$ and 
   $\kappa^{\alpha}_{+}(x):=(b-x)^{\alpha-1}$ ($0<\alpha<1$) do not vanish in the 
   latter case. Since both kernel
   functions are singular now, they must be ``removed" from any function $u\in {^{\pm}}{W}{^{\alpha,p}}(\Omega)$
   in order to obtain the desired Sobolev and Poincar\'e inequalities for $u$.

   \begin{theorem}
   	Let $0<\alpha<1$ and $1 \leq  p < \frac{1}{\alpha}$.
   	Then there exists a constant $C = C(\Omega , \alpha, p) > 0$ such that
   	\begin{align}\label{FractionalSobolevInequalityOmega}
   	\|u - c_{\pm}^{1-\alpha} \kappa_{\pm}^{\alpha} \|_{L^{r}(\Omega)} \leq \|{^{\pm}}{\mathcal{D}}{^{\alpha}} u\|_{L^p(\Omega)} \qquad \forall \, 1 \leq r \leq p^*.
   	\end{align}
   \end{theorem}
   
   \begin{proof}
   	It follows by the $L^p$ mapping properties of the fractional integrals 
   	(cf. \cite[Theorem 2.6]{Feng_Sutton}) and the FTwFC (cf. Theorem \ref{FTWFC}) 
   	that 
   	\begin{align*}
   	\|u - c_{\pm}^{1-\alpha} \kappa^{\alpha}_{\pm} \|_{L^{p^*}(\Omega)} = \| {^{\pm}}{I}{^{\alpha}} {^{\pm}}{\mathcal{D}}{^{\alpha}} u\|_{L^{p^*}(\Omega)}  
   	\leq C \| {^{\pm}}{\mathcal{D}}{^{\alpha}} u \|_{L^{p}(\Omega)}.
   	\end{align*}
   	Since $\Omega=(a,b)$ is finite,  the desired inequality \eqref{FractionalSobolevInequalityOmega} follows from 
   	the above inequality and an application of H\"older's inequality. The proof is complete. 
   \end{proof}
   
   \begin{remark}
   	An important consequence of the above theorem is that it illustrates the need for $u \in L^{\mu}(\Omega)$ with $\mu > p^*$
   	in the extension theorem (cf. Theorem \ref{ExteriorExtension}) because the fractional Sobolev spaces 
   	${^{\pm}}{W}{^{\alpha,p}}(\Omega)$ may not embed into $L^{\mu}(\Omega)$ for $\mu > p^*$ in general.
   \end{remark}
   
   Repeating the first part of the above proof (with slight modifications),  we can easily show the 
   	following Poincar\'e inequality in fractional order spaces ${^{\pm}}{W}{^{\alpha,p}}(\Omega)$. 
   
   \begin{theorem}\label{thm_PoincareI}
   	Fractional Poincar\'e - Let $0 < \alpha <1$ and $1 \leq p < \infty$. Then there exists a constant $C = C(\alpha, \Omega)>0$ such that
   	\begin{align}\label{FractionalPoincare}
   	\|u - c_{\pm}^{1-\alpha}\kappa_{\pm}^{\alpha}\|_{L^{p}(\Omega)} \leq C\|{^{\pm}}{\mathcal{D}}{^{\alpha}} u\|_{L^{p}(\Omega)}
   	\qquad \forall u \in {^{\pm}}{W}{^{\alpha,p}}(\Omega).
   	\end{align}
   \end{theorem}
   
   It is worth noting that no restriction on $p$ with respect to $\alpha$ is imposed in 
   Theorem \ref{thm_PoincareI} because no embedding result for fractional integrals 
   is used in the proof. The equation \eqref{FractionalPoincare} is the fractional analogue to the well known Poincar\'e inequality 
   \begin{align}\label{Poincare}
       \|u - u_{\Omega}\|_{L^{p}(\Omega)} \leq C \| \mathcal{D} u \|_{L^{p}(\Omega)} \qquad \forall u \in W^{1,p}(\Omega)
   \end{align}
   where $u_{\Omega} : = |\Omega|^{-1}\int_{\Omega} u \,dx$ (\cite{Evans}). In the space $W^{1,p}(\Omega)$, a specific kernel function (a constant, i.e., $u_\Ome$), that depends on $u$, is subtracted from the function $u$. In (\ref{FractionalPoincare}), the analogue to this constant kernel 
   function, which must be subtracted from $u$, is $c_{\pm}^{1-\alpha} \kappa_{\pm}^{\alpha}$, where the   dependence on $u$ is hidden in $c_{\pm}^{1-\alpha}$. 
   
   Moreover, to obtain a fractional analogue to the traditional Poincar\'e inequality 
   \begin{align} 
   \|u \|_{L^{p}(\Omega)} \leq C \|\mathcal{D}u\|_{L^{p}(\Omega)} \qquad \forall u \in W^{1,p}_{0}(\Omega),
   \end{align}
   we have two options. The first one is to simply impose the condition $u \in {^{\pm}}{\mathring{W}}{^{\alpha,p}}(\Omega)$ (see \eqref{mathring}).
   It is an easy corollary of Theorem \ref{thm_PoincareI} that 
   \begin{align}\label{SimplePoincare}
       \|u\|_{L^{p}(\Omega)} \leq C \| {^{\pm}}{\mathcal{D}}{^{\alpha}} u \|_{L^{p}(\Omega)} \qquad \forall u \in {^{\pm}}{\mathring{W}}{^{\alpha,p}}(\Omega).
   \end{align}
   From the perspective of Poincar\'e inequalities, this condition is 
   comparable to a mean-zero condition imposed on the Sobolev space $W^{1,p}(\Omega)$. 
   In order to establish the second set of conditions under which the estimate \eqref{SimplePoincare}
   can hold, we first need to establish the following lemma.
  
   \begin{lemma}\label{NoConstants}
   Let $\Omega = (a,b)$ and $0 < \alpha <1$.
    If $u \in {W}{^{\alpha,p}}(\Omega)$, then $c^{1-\alpha}_{+} := {^{+}}{I}{^{\alpha}} u (b) = 0$ and $c^{1-\alpha}_{-} : = {^{-}}{I}{^{\alpha}} u (a) = 0$.
    \end{lemma}

\begin{proof}
    Let $u \in {W}^{\alpha,p}(\Omega)$. It follows that $u \in C(\overline{\Omega})$. Then a quick calculation yields  
    \begin{align*}
        c^{1-\alpha}_{-} =  \lim_{x \rightarrow a} \left|\int_{a}^{x} \dfrac{u(y)}{(x-y)^{1- \alpha}} \,dy \right|
        &\leq \lim_{x\rightarrow a}\|u\|_{L^{\infty}(\Omega)} (x-a)^{\alpha} 
        = 0.
    \end{align*}
    A similar calculation can be done for $c_{+}^{1-\alpha}$. The proof is complete.
\end{proof}

Now we can formalize the desired Poincar\'e inequality.

\begin{theorem}\label{SymmetricFractionalPoincare}
    Let $\Omega \subset \R$, $0 < \alpha <1$, and $1 < p < \infty$. Then there exists a constant $C = C(\alpha, \Omega)>0$ such that
    \begin{align}\label{SymmetricPoincare}
        \|u\|_{L^{p}(\Omega)} \leq C \| {^{\pm}}{\mathcal{D}}{^{\alpha}} u \|_{L^{p}(\Omega)} \qquad \forall u \in {W}^{\alpha,p}(\Omega).
    \end{align}
\end{theorem}

\begin{proof}
    The proof follows as a direct consequence of the FTwFC (cf. Theorem \ref{FTWFC}), 
    Lemma \ref{NoConstants}, and the stability estimate for fractional integrals.
\end{proof}

Another question may come to mind is whether such an estimate can be established in 
the one-sided zero-trace spaces ${^{\pm}}{W}{^{\alpha,p}_{0}}(\Omega)$. 
Functions belonging to ${^{\pm}}{W}{^{\alpha,p}_{0}}(\Omega)$ do not 
  guarantee that $c^{1-\alpha}_{\pm} = 0$. Hence, ${^{\pm}}{W}{^{\alpha,p}_{0}}(\Omega) \not\subset  {^{\pm}}{\mathring{W}}{^{\alpha,p}}(\Omega)$ and such an inequality does not hold in ${^{\pm}}{W}{^{\alpha,p}_{0}}(\Omega)$ in general.

  \subsection{The Dual Spaces ${^{\pm}}{W}{^{-\alpha,q}}(\Ome)$ and 
  	$W^{-\alpha ,q}(\Ome)$}\label{sec-4.5}
   
   In this subsection we assume that $1 \leq p < \infty$ and $1 < q \leq \infty$ 
   so that $1/p + 1/q =1$. 
   \begin{definition}
       We denote ${^{\pm}}{W}{^{-\alpha , q}}(\Omega)$ as the dual space of ${^{\pm}}{\mathring{W}}{^{\alpha,p}_{0}}(\Omega)$ and $W^{-\alpha , q}(\Omega)$ as the dual space of $W^{\alpha,p}_{0}(\Omega)$. When $p=2$, we set  ${^{\pm}}{H}{^{-\alpha}}(\Omega):={^{\pm}}{\mathring{W}}{^{-\alpha , 2}_{0}}(\Omega)$ and $H^{-\alpha}(\Omega):=W^{-\alpha , 2}(\Omega)$. 
   \end{definition} 
   
   It is our aim to fully characterize these spaces; as is well known in the case of integer order Sobolev dual spaces, $W^{-1,q}(\Omega)$ (cf. \cite{Brezis}), in particular, for $q=2$.
   
   We will begin with the symmetric spaces since the presentation is more natural and easily understood than that for the one-sided spaces. It is a consequence of Proposition \ref{ConstantZero} and Proposition \ref{ZeroTraceNorm} that $$W^{\alpha,p}_{0}(\Omega) \subset L^{p}(\Omega) \subset W^{-\alpha , q}(\Omega)$$ where these injections are continuous for $1 \leq p < \infty$ and dense for $1 < p < \infty$ since $W^{\alpha,p}_{0}(\Omega)$ and $L^{p}(\Omega)$ are reflexive in this range. In order to formally characterize the elements of $W^{-\alpha, q}(\Omega)$, we present the following theorem. 
   
   \begin{theorem}
   Let $F \in W^{-\alpha ,q}(\Omega)$. Then there exists three functions, $f_0, f_1, f_2 \in L^{q}(\Omega)$ such that 
   \begin{align}\label{SymmetricDualCharacterization}
       \langle F , u \rangle  = \int_{\Omega} f_0 u\, dx + \int_{\Omega} f_1 {^{-}}{\mathcal{D}}{^{\alpha}} u \,dx + \int_{\Omega} f_2 {^{+}}{\mathcal{D}}{^{\alpha}} u\, dx \qquad \forall \, u \in W^{\alpha , p }_{0}(\Omega)
   \end{align}
   and 
   \begin{align}\label{DualNorm}
       \|F\|_{W^{-\alpha, q}(\Omega)} = \max \Bigl\{ \|f_0\|_{L^{q}(\Omega)},\|f_1\|_{L^{q}(\Omega)}, \|f_2\|_{L^{q}(\Omega)} \Bigr\}.
   \end{align}
   When $\Omega \subset \R$ bounded, we can take $f_0 = 0$.
   \end{theorem}
   
   \begin{proof}
   Consider the product space $E = L^{p}(\Omega) \times L^{p}(\Omega) \times L^{p}(\Omega)$ equipped with the norm 
  \[
  \|h\|_{E} = \|h_0\|_{L^{p}(\Omega)} + \|h_1\|_{L^{p}(\Omega)} + \| h_2\|_{L^{p}(\Omega)},
  \]
   where $h = [ h_0 , h_1 , h_2]$. The map $T: W^{\alpha , p}_{0}(\Omega) \rightarrow E$ 
   defined by 
   \[
   T(u) = [u, {^{-}}{\mathcal{D}}{^{\alpha}} u , {^{+}}{\mathcal{D}}{^{\alpha}} u]
   \]
   is an isometry from $W^{\alpha , p }_{0}(\Omega)$ into $E$. Given the space 
   $(G, \|\cdot\|_{E})$ be the image of $W^{\alpha,p}_{0}$ under $T$ 
   ($G= T(W^{\alpha,p}_{0}(\Omega)$)) and $T^{-1} : G \rightarrow W^{\alpha,p}_{0}(\Omega)$. 
   Let $F \in W^{-\alpha ,q}(\Omega)$ be a continuous linear functional on $G$ defined by 
   $F(h) = \langle F, T^{-1} h\rangle$. By the Hahn-Banach theorem, it can be extended to 
   a continuous linear functional $S$ on all of $E$ with $\|S\|_{E^{*}} = \|F\|$. 
   By the Riesz representation theorem, we know that there exists three functions 
   $f_0 , f_1 , f_2 \in L^{q}(\Omega)$ such that 
   \begin{align*}
       \langle S, h\rangle = \int_{\Omega} f_0 h_0 \, dx + \int_{\Omega} f_1 h_1\, dx + \int_{\Omega} f_2 h_2\, dx \qquad \forall \, h = [h_0 , h_1 , h_2] \in E.
   \end{align*}
   Moreover, we have 
   \begin{align*}
       \dfrac{|\langle S , h \rangle |}{\|h\|_{E}} &= \dfrac{1}{\|h\|_{E}}\left| \int_{\Omega} f_0 h_0\, dx + \int_{\Omega} f_1 h_1 \, dx + \int_{\Omega} f_2 h_2 \, dx \right| \\ 
       &\leq \dfrac{1}{\|h\|_{E}} \Bigl( \|f_{0}\|_{L^{q}(\Omega)} \| h_0 \|_{L^{p}(\Omega)} + \|f_{1}\|_{L^{q}(\Omega)} \|h_1\|_{L^{p}(\Omega)} + \|f_2\|_{L^{q}(\Omega)} \|h_2\|_{L^{p}(\Omega)} \Bigr)\\
       &\leq \max \Bigl\{ \|f_0\|_{L^{p}(\Omega)}, \|f_1\|_{L^{p}(\Omega)} , \|f_2\|_{L^{p}(\Omega)} \Bigr\}.
   \end{align*} 
   Upon taking the supremum, we are left with 
   \[ 
   \|S\|_{E^{*}} = \max \Bigl\{ \|f_0\|_{L^{p}(\Omega)}, \|f_1\|_{L^{p}(\Omega)} ,
    \|f_2\|_{L^{p}(\Omega)} \Bigr\}.
   \]
   Furthermore, we have 
   \begin{align*}
       \langle S , Tu \rangle = \langle F, u \rangle = \int_{\Omega} f_0 u\, dx + \int_{\Omega} f_1 {^{-}}{\mathcal{D}}{^{\alpha}} u \, dx + \int_{\Omega} f_2 {^{+}}{\mathcal{D}}{^{\alpha}} u\, dx \quad \forall\, u \in W^{\alpha, p}_{0}(\Omega).
   \end{align*}
   
   When $\Omega$ is bounded, recall that $\|u \|_{W^{\alpha,p}_{0}} = \bigl(\|{^{-}}{\mathcal{D}}{^{\alpha}} u \|_{L^{p}(\Omega)}^{p} + \| {^{+}}{\mathcal{D}}{^{\alpha}} u \|_{L^{p}(\Omega)}^{p} \bigr)^{1/p}$. 
   Then we can repeat the same argument with $E = L^{p}(\Omega) \times L^{p}(\Omega)$ and $T(u) = [ {^{-}}{\mathcal{D}}{^{\alpha}} u , {^{+}}{\mathcal{D}}{^{\alpha}} u ]$. The proof is complete.
   \end{proof}
   
   \begin{remark}(a) The functions $f_0 , f_1 , f_2$ are not uniquely determined by $F$. 
   
   (b) We write $F = f_0 + {^{+}}{\mathcal{D}}{^{\alpha}} f_1 + {^{-}}{\mathcal{D}}{^{\alpha}} f_2$ whenever (\ref{SymmetricDualCharacterization}) holds. Formally, this is a consequence of integration by parts in the right hand side of (\ref{SymmetricDualCharacterization}).
   
   (c) The first assertion of Proposition \ref{SymmetricDualCharacterization} also holds for continuous linear functionals on $W^{\alpha ,p}(\Omega)$ ($1 \leq p < \infty)$. That is, for every $F \in (W^{\alpha,p}(\Omega))^{*}$, 
   \[
   \langle F , u \rangle  = \int_{\Omega} f_0 u \,dx + f_1 {^{-}}{\mathcal{D}}{^{\alpha}} u \, dx
    + f_2 {^{+}}{\mathcal{D}}{^{\alpha}} u\, dx \qquad \forall\, u \in W^{\alpha ,p}(\Omega)
    \]
    for some functions $f_0 , f_1 , f_2 \in L^{q}(\Omega)$. 
   \end{remark}
   
   Of course, the above results also hold for functions in $H^{-\alpha}(\Omega)$. However, 
   in this case, the use of the inner product and Hilbert space structure allows for improved presentation and richer characterization. We state them in the following proposition. 
   
   \begin{proposition}\label{SymmetricDualCharacterization1}
    Let $F \in H^{-\alpha}(\Omega)$. Then \begin{align}\label{DualNorm1}
        \|F\|_{H^{-\alpha}(\Omega)} = \inf \left\{ \left(\int_{\Omega} \sum_{i=0}^{2} |f_i|^{2} \,dx \right)^{\frac12};\, f_0 , f_1 , f_2 \in L^2(\Omega) \text{ satisfy } (\ref{SymmetricDualCharacterization}) \right\}.
    \end{align}
   \end{proposition}
   
   \begin{proof}
   We begin with an altered proof of (\ref{SymmetricDualCharacterization}) for the special case $p=2$. Not only is the proof illustrative, but we will also refer to components of it to prove necessary assertions of this proposition. 
   
   For any $u , v \in H^{\alpha}_{0}(\Omega)$, we define the inner product 
   \[
   (u,v) = \int_{\Omega} \bigl( uv + {^{-}}{\mathcal{D}}{^{\alpha}} u {^{-}}{\mathcal{D}}{^{\alpha}} v + {^{+}}{\mathcal{D}}{^{\alpha}} u {^{+}}{\mathcal{D}}{^{\alpha}} v \bigr)\,dx.
   \]
   Given $F \in H^{-\alpha}(\Omega)$, it follows from Riesz Representation theorem 
   that there exists a \textit{unique} $u \in H^{\alpha}_{0}(\Omega)$ so that 
   $\langle F , v \rangle = (u,v)$ for all $v \in H^{\alpha}_{0}(\Omega)$; that is 
   \begin{align}\label{Riesz1}
       \langle F , v \rangle =\int_{\Omega} \bigl( uv + {^{-}}{\mathcal{D}}{^{\alpha}} u {^{-}}{\mathcal{D}}{^{\alpha}} v + {^{+}}{\mathcal{D}}{^{\alpha}} u {^{+}}{\mathcal{D}}{^{\alpha}} v \bigr) \,dx \qquad \forall\, v \in H^{\alpha}_{0}(\Omega).
   \end{align}
   Taking 
   \begin{align}\label{Riesz2}
     f_0 = u,\qquad f_1 = {^{-}}{\mathcal{D}}{^{\alpha}} u,\qquad
      f_2 = {^{+}}{\mathcal{D}}{^{\alpha}} u,
   \end{align}
   then (\ref{SymmetricDualCharacterization}) holds. 
   
   It follows by (\ref{SymmetricDualCharacterization}) that there exists $g_0 , g_1 , g_2 \in L^{2}(\Omega)$ so that 
   \begin{align}\label{Riesz3}
       \langle F , v \rangle = \int_{\Omega}  \bigl(g_0 v + g_1 {^{-}}{\mathcal{D}}{^{\alpha}} v + g_2 {^{+}}{\mathcal{D}}{^{\alpha}} v \bigr)\, dx \qquad\forall v \in H^{\alpha}_{0}(\Omega).
   \end{align}
   Thus, taking $v = u$ in (\ref{Riesz1}) and combing that with (\ref{Riesz2}) and (\ref{Riesz3})
   yield 
   \begin{align*}
       \int_{\Omega} f_0^2 + f_1^2 + f_2^2 &= \int_{\Omega} \bigl(u^2 + ({^{-}}{\mathcal{D}}{^{\alpha}} u)^{2} + ( {^{+}}{\mathcal{D}}{^{\alpha}} u)^{2}\bigr)\,dx \\
       &= \int_{\Omega} \bigl( g_0 u + g_1 {^{-}}{\mathcal{D}}{^{\alpha}} u + g_2 {^{+}}{\mathcal{D}}{^{\alpha}} u \bigr) \, dx 
       \leq \int_{\Omega} \bigl( g_0^2 + g_1 ^2 + g_2^2\bigr)\, dx.
   \end{align*}
   It follows from (\ref{SymmetricDualCharacterization}) and the dual norm definition that for $\|v\|_{H^{\alpha}_{0}(\Omega)} \leq 1$, 
   \begin{align*}
       \|F\|_{H^{-\alpha}(\Omega)} \leq \left(\int_{\Omega} \bigl(f_0^2 + f_1^2 + f_2^2 \bigr)\, dx \right)^{\frac12}.
   \end{align*}
   Setting $v = u/\|u\|_{H^{\alpha}_{0}(\Omega)}$ in (\ref{Riesz1}), we deduce that 
   \begin{align*}
       \|F\|_{H^{-\alpha}(\Omega)} = \left( \int_{\Omega}  \bigl(f_0^2 + f_1^2 + f_2^2 \bigr)\, dx \right)^{\frac12}.
   \end{align*}
   Therefore, (\ref{DualNorm1}) must hold. The proof is complete.
   \end{proof}
   
   \begin{remark}
   Similar to the integer order case, we define the action of $v \in L^{2}(\Omega) \subset H^{-\alpha}(\Omega)$ on any $u \in H^{\alpha}_{0}(\Omega)$ by
    \begin{align}
        \langle v,u \rangle = \int_{\Omega} vu\,dx.
    \end{align}
    That is to say that given $v \in L^{2}(\Omega) \subset H^{-\alpha}(\Omega)$, we associate it with the bounded linear functional $v: H^{\alpha}_{0}(\Omega) \rightarrow \R$ defined by $\langle v , u \rangle = v(u) = \int_{\Omega} vu$. It is easy to check that this mapping is in fact continuous/bounded on $H^{\alpha}_{0}(\Omega)$. 
   \end{remark}
   
   \medskip
   Now we turn our attention to dual spaces of one-sided Sobolev spaces. The situation in this case is more complicated. This is due to the fact that there are several variations of the parent spaces ${^{\pm}}{W}{^{\alpha,p}}$ of which we might consider. To be specific, we consider 
   a space $W$ where 
   \[
   W \in \bigl\{ {^{\pm}}{W}{^{\alpha,p}}, {^{\pm}}{W}{^{\alpha,p}_{0}}, {^{\pm}}{\mathring{W}}{^{\alpha,p}}, {^{\pm}}{\mathring{W}}{^{\alpha,p}_{0}} \bigr\}.
   \]
    Thus, we want to know which of these spaces produces a dual space that can be characterized 
    in similar fashion as for the symmetric space $W^{\alpha,p}_{0}$.
   
   To answer this question, we first proposed that in order to prove a rich characterization of dual spaces, we must first have the continuous and dense inclusion $W \subset L^{p} \subset W^{*}$ for appropriate ranges of $p$. Effectively, it is necessary to have the inequality $\|u \|
   _{L^{p}} \leq C \|u\|_{W}$ for every $u \in W$. More or less, this question is informed by the existence of fractional Poincar\'e inequalities in $W$. It is known that in general, $\|u \|_{L^{p}(\Omega)} \not\leq C \| {^{\pm}}{\mathcal{D}}{^{\alpha}} u \|_{L^{p}(\Omega)}$ 
   for every $u \in {^{\pm}}{W}{^{\alpha,p}}(\Omega)$ and $u\in {^{\pm}}{W}{^{\alpha,p}_{0}}$, 
   and note  ${^{\pm}}{W}{^{\alpha,p}_{0}}(\Omega) \not\hookrightarrow {^{\pm}}{\mathring{W}}{^{\alpha,p}}(\Omega)$). For these reasons, we are left to characterize the dual space ${^{\pm}}{W}{^{-\alpha, q}}(\Omega) : = ({^{\pm}}{\mathring{W}}{^{\alpha,p}}(\Omega))^{*}$.
   
   It is easy to see that there holds 
   \begin{align}
       {^{\pm}}{\mathring{W}}{^{\alpha,p}}(\Omega) \subset L^{p}(\Omega) \subset {^{\pm}}{W}{^{-\alpha,q}}(\Omega),
   \end{align}
   where the injections are continuous for $1 \leq p < \infty$ and dense for $1 < p < \infty$ since ${^{\pm}}{\mathring{W}}{^{\alpha,p}}(\Omega)$ is reflexive in this range. 
   
   Now we are well equipped to characterize ${^{\pm}}{W}{^{-\alpha,q}}(\Omega)$. For brevity, we will state the results and omit the proofs since each of them can be done using the same techniques as used in the symmetric case for the spaces $W^{-\alpha,q}(\Ome)$ and $H^{-\alpha}(\Ome)$. 
   
   \begin{theorem}
   Let $F \in {^{\pm}}{W}{^{-\alpha,q}}(\Omega)$. Then there exists two functions, $f_0 , f_1 \in L^{q}(\Omega)$ such that 
   \begin{align}\label{DualCharacterization}
    \langle F, u \rangle _{\pm} = \int_{\Omega} \bigl( f_0 u + f_1 {^{\pm}}{\mathcal{D}}{^{\alpha}} u \bigr)\, dx \qquad \forall \, u \in {^{\pm}}{\mathring{W}}{^{\alpha,p}}(\Omega)
   \end{align}
   and 
   \begin{align}
       \|F\|_{{^{\pm}}{W}{^{-\alpha , q}}(\Omega)} = \max \Bigl\{ \|f_0\|_{L^{q}(\Omega)} , \|f_1\|_{L^{q}(\Omega)} \Bigr\}.
   \end{align}
   \end{theorem}
   
   \begin{proposition}
  Let $F \in {^{\pm}}{H}{^{-\alpha}}(\Omega)$. Then 
   \begin{align}
       \|F\|_{{^{\pm}}{H}{^{-\alpha}}(\Omega)} = \inf \left\{ \left( \int_{\Omega} \bigl(f_0^2 + f_1^2 \bigr)\, dx  \right)^{\frac12};\, f_0 , f_1 \in L^{2}(\Omega)\text{ satisfying } (\ref{DualCharacterization}) \right\}. 
   \end{align}
   \end{proposition}
   
   \begin{remark}
   Similar to the symmetric case, we define the action of $v \in L^{2}(\Omega) \subset {^{\pm}}{H}{^{-\alpha}}(\Omega)$ on any $u \in {^{\pm}}{\mathring{H}}{^{\alpha}}(\Omega)$ by 
   \begin{align}
       \langle v , u \rangle = \int_{\Omega} vu \,dx.
   \end{align}
   \end{remark}
           
    \subsection{Relationships Between Fractional Sobolev Spaces}\label{sec-4.6}
    In this subsection we establish a few connections between the newly defined fractional Sobolev
    spaces ${^{\pm}}{W}{^{\alpha,p}}(\Omega)$ and ${W}^{\alpha,p}(\Omega)$ with some existing fractional Sobolev spaces recalled in Section \ref{sec-5.1}. Before doing that, we first address the issues of their consistency over subdomains, inclusivity across orders of differentiability, and their consistency with the existing integer order Sobolev spaces.
    
    \begin{proposition}
        Let $\Omega = (a,b)$, $0 < \alpha < \beta < 1$ and $1 \leq p < \infty$. If $u \in {^{\pm}}{W}{^{\beta,p}}(\Omega)$, then $ u\in {^{\pm}}{W}{^{\alpha,p}}(\Omega)$. 
    \end{proposition}
    
    \begin{proof}
        By the FTwFC (cf. Theorem \ref{FTWFC}), we have
            $u = c^{1-\beta}_{\pm} \kappa^{\beta}_{\pm} + {^{\pm}}{I}{^{\beta}} {^{\pm}}{\mathcal{D}}{^{\beta}}u$
        and by the inclusivity result for weak fractional derivatives, 
        ${^{\pm}}{\mathcal{D}}{^{\alpha}}u$ exists and is given by 
        \begin{align*}
            {^{\pm}}{\mathcal{D}}{^{\alpha}} u &= c^{1-\beta}_{\pm} \kappa^{\beta - \alpha}_{\pm} + {^{\pm}}{I}{^{\beta - \alpha}} {^{\pm}}{\mathcal{D}}{^{\beta}} u  \\
            &= c^{1-\beta}_{\pm} \kappa^{\beta}_{\pm} \kappa^{-\alpha}_{\pm} + {^{\pm}}{I}{^{\beta - \alpha}}{^{\pm}}{\mathcal{D}}{^{\beta}} u \\ 
            &= ( u - {^{\pm}}{I}{^{\beta}} {^{\pm}}{\mathcal{D}}{^{\beta}} u ) \kappa^{-\alpha}_{\pm} + {^{\pm}}{I}{^{\beta - \alpha}} {^{\pm}}{\mathcal{D}}{^{\beta}} u.
        \end{align*}
        It follows by direct estimates that there exists $C  = C(\Omega , \alpha ,\beta , p)$ so that 
        \begin{align*}
            \|{^{\pm}}{\mathcal{D}}{^{\alpha}} u \|_{L^{p}(\Omega)} \leq C \| u \|_{{^{\pm}}{W}{^{\beta,p}}(\Omega)}. 
        \end{align*}
        The proof is complete. 
    \end{proof}
    
    \begin{remark}
        This inclusivity property is trivial in the integer order Sobolev spaces, but may not be so in 
        fractional Sobolev spaces, which may be a reason why it has not been discussed in the literature.  
        However, in our case, the proof is not difficult thanks to the FTwFC. 
    \end{remark}
    
    Unlike the integer order case, the consistency on subdomains is more difficult to establish 
    in the spaces ${^{\pm}}{W}{^{\alpha,p}}$. The following proposition and its accompanying proof 
    provide further insight to the effect of domain-dependent derivatives and their associated kernel functions.
    
    \begin{proposition}
        Let $\Omega =(a,b)$, $\alpha > 0$, $1 < p < \infty$, $\mu > p(1- \alpha p)^{-1}$.
        Suppose that $u \in {^{\pm}}{W}{^{\alpha,p}}(\Omega) \cap L^{\mu}(\Omega)$. Then for any $\Omega'=(c,d) \subset \Omega$, $u \in {^{\pm}}{W}{^{\alpha,p}}(\Omega')$.
    \end{proposition}
    
    \begin{proof}
         Since $(c,d) \subset (a,b)$, it is easy to see that $\|u \|_{L^{p}((c,d))} \leq \|u \|_{L^{p}((a,b))}$. Thus we only need to show that $u$ has a weak derivative on $(c,d)$ that belongs to $L^{p}((c,d))$. 
         
         Choose $\{u_j\}_{j=1}^{\infty} \subset C^{\infty}((a,b))$ so that $u_j \rightarrow u$ in ${^{\pm}}{W}{^{\alpha,p}}((a,b))$. It follows that $u_j \in C^{\infty}([c,d])$ and for any $\varphi \in C^{\infty}_{0}((c,d))$ there holds for the left derivative
        \begin{align*}
            \int_{c}^{d} u {^{+}}{D}{^{\alpha}} \varphi\, dx = \lim_{ j \rightarrow \infty} \int_{c}^{d} u_j {_{x}}{D}{^{\alpha}_{d}}\varphi\, dx 
            = \lim_{j \rightarrow \infty} \int_{c}^{d} {_{c}}{D}{^{\alpha}_{x}} u_j \varphi\, dx.
        \end{align*}
        Then we want to show that there exists $v \in L^{p}((c,d))$ such that 
        \begin{align*}
            \lim_{j \rightarrow \infty} \int_{c}^{d} {_{c}}{D}{^{\alpha}_{x}} u_j \varphi\, dx = \int_{c}^{d} v \varphi\, dx. 
        \end{align*}
        Note that 
            $ {_{c}}{D}{^{\alpha}_{x}}u_j(x) = {_{a}}{D}{^{\alpha}_{x}} u_j(x) - {_{a}}{D}{^{\alpha}_{c}} u_j(x).$
        Using this decomposition, we get
        \begin{align*}
            &\|{_{c}}{D}{^{\alpha}_{x}}u_j \|_{L^{p}((c,d))}^{p} = \| {_{a}}{D}{^{\alpha}_{x}} u_j - {_{a}}{D}{^{\alpha}_{c}} u_j\|_{L^{p}((c,d))}^{p} \\
            &\quad\leq \| {_{a}}{D}{^{\alpha}_{x}} u_j\|_{L^{p}((a,b))}^{p} + \int_{c}^{d} \left|\int_{a}^{c} \dfrac{u_j(y)}{(x-y)^{1+\alpha}}\,dy \right|^{p}\,dx \\ 
            &\quad\leq \| {_{a}}{D}{^{\alpha}_{x}} u_j\|_{L^{p}((a,b))}^{p}+  \|u_j\|_{L^{\mu}((a,c))}^{p} \int_{c}^{d} \left(\int_{a}^{c} \dfrac{dy}{(x-y)^{\nu (1+\alpha)}}\right)^{\frac{p}{\nu}} \,dx \\ 
            &\quad\leq  \| {_{a}}{D}{^{\alpha}_{x}} u_j\|_{L^{p}((a,b))}^{p}+  \|u_j\|_{L^{\mu}((a,c))}^{p} \int_{c}^{d} (x-a)^{\frac{p}{\nu} - p (1+\alpha)} + (x-c)^{\frac{p}{\nu} -p(1+\alpha)}\,dx,
        \end{align*}
        which is bounded if and only if $\mu > p(1-\alpha p)^{-1}$. Choosing $j$ sufficiently large, we have that the sequence ${_{c}}{D}{^{\alpha}_{x}} u_j $ is bounded in $L^{p}((c,d))$. Therefore, there exists a function $v \in L^p((c,d))$ and a subsequence (still denoted by ${_{c}}{D}{^{\alpha}_{x}} u_j$) so that ${_{c}}{D}{^{\alpha}_{x}} u_j \rightharpoonup v$. It follows that 
        \begin{align*}
           \int_{c}^{d} u {^{+}}{D}{^{\alpha}}\varphi\, dx =  \lim_{j \rightarrow \infty} \int_{c}^{d} {_{c}}{D}{^{\alpha}_{x}} u_j \varphi \,d x= \int_{c}^{d} v \varphi\, dx.
        \end{align*}
        Hence $u \in {^{-}}{W}{^{\alpha,p}}((c,d))$. Similarly, we can prove
        that the conclusion also holds for the right derivative.   The proof is complete.
    \end{proof}
    
    \subsubsection{\bf Consistency with $W^{1,p}(\Omega)$}\label{sec-4.6.1}
         Our aim here is to show that there exists a consistency between our newly defined fractional Sobolev spaces 
         and the integer order Sobolev spaces. To the end,  we need to show that there is a consistency between 
         fractional order weak derivatives and integer order weak derivatives, which is detailed in the   
         lemma below. 
        
        \begin{lemma}\label{lem_consistency}
            Let $\Omega \subseteq \R$, $0 <\alpha<1$ and $1 \leq p < \infty$. Suppose $u \in W^{1,p}(\Omega)$. Then for every $\psi \in C^{\infty}_{0}(\Omega)$, ${^{\pm}}{\mathcal{D}}{^{\alpha}} \psi(u)= -{^{\pm}}{I}{^{1- \alpha}} [ \psi'(u) Du] \in L^{p}(\Omega)$. 
        \end{lemma}
    
        \begin{proof}
            Let $u \in W^{1,p}(\Omega) \cap {^{\pm}}{W}{^{\alpha,p}}(\Omega)$. By the density/approximation theorem, there exists $\{u_j\}_{j=1}^{\infty} \subset C^{\infty}(\Omega)$ 
            such that $u_j \rightarrow u$ in $W^{1,p}(\Omega)$. Then we have  
            \begin{align*}
                \int_{\Omega} \psi(u) {^{\mp}}{D}{^{\alpha}}\varphi\, dx = \lim_{j\rightarrow \infty} \int_{\Omega} \psi (u_j) {^{\mp}}{D}{^{\alpha}}\varphi \, dx
                &= \lim_{j \rightarrow \infty} (-1) \int_{\Omega} \psi' (u_j) Du_{j} {^{\mp}}{I}{^{1-\alpha}} \varphi\, dx \\ 
                &= \lim_{j \rightarrow \infty} (-1)\int_{\Omega}  {^{\pm}}{I}{^{1- \alpha}} [ \psi'(u_j) Du_j] \varphi \, dx.
            \end{align*}
           
           Next, we claim that ${^{\pm}}{I}{^{1-\alpha}} [ \psi ' (u_j) Du_j] \rightarrow {^{\pm}}{I}{^{1-\alpha}} [ \psi'(u)\mathcal{D}u]$ in $L^{p}(\Omega)$ where $\mathcal{D}$ denotes the integer weak derivative. Our claim follows because 
            \begin{align*}
                \|{^{\pm}}{I}{^{1-\alpha}} [\psi'(u_j) D u_j] - {^{\pm}}{I}{^{1-\alpha}} [ \psi'(u)\mathcal{D}u]\|_{L^{p}(\Omega)} &\leq C \| \psi'(u_j) Du_j - \psi'(u)\mathcal{D}u\|_{L^{p}(\Omega)}
            \end{align*}
            which converges to zero by the assumptions on $\psi$ and on $\{u_j\}_{j=1}^{\infty}$ and the chain rule in $W^{1,p}(\Omega)$.
            The proof is complete. 
        \end{proof}
        
        \begin{remark}
        	The identity ${^{\pm}}{\mathcal{D}}{^{\alpha}} \psi(u)= -{^{\pm}}{I}{^{1- \alpha}} [ \psi'(u) \mathcal{D}u] \in L^{p}(\Omega)$ can be regarded as a special fractional chain rule, which also explains why there is no clean fractional chain rule in general. 
        \end{remark}

        Our first consistency result will be one that allows us to make no assumption on the relationship 
        between $\alpha$ and $p$. However, a restriction on function spaces must be imposed, which will be shown 
        later to be a price to pay without imposing any restriction on the relationship between $\alpha$ and $p$.
        
        \begin{theorem}\label{TraceZeroConsistency}
             Let $\Omega \subset \R$, $0 <\alpha<1$ and $1 \leq p < \infty$. Then $W^{1,p}_{0}(\Omega) \subset {^{\pm}}{W}{^{\alpha,p}}(\Omega)$. Hence, $W^{1,p}_0(\Omega) \subset   {W}^{\alpha,p}(\Omega)$.
        \end{theorem}
    
        \begin{proof}
            Let $u \in W^{1,p}_{0}(\Omega)$. By the density/approximation theorem, there exists $\{u_j\}_{j=1}^{\infty} \subset C^{\infty}_{0}(\Omega)$ such that $u_j \rightarrow u$ in $W^{1,p}(\Omega)$. Then we have 
            \begin{align*}
                \int_{\Omega} u {^{\mp}}{D}{^{\alpha}}\varphi\, dx = \lim_{j \rightarrow \infty} \int_{\Omega} u_j {^{\mp}}{D}{^{\alpha}}\varphi \, dx
                &= \lim_{j\rightarrow \infty} (-1)\int_{\Omega}  Du_j {^{\mp}}{I}{^{1-\alpha}}\varphi \, dx \\
                &= \lim_{j \rightarrow \infty} (-1) \int_{\Omega}  {^{\pm}}{I}{^{1-\alpha}} Du_j \varphi \,dx .
            \end{align*}
            Next, by the boundedness of ${^{\pm}}{I}{^{1-\alpha}}$ we get 
            \begin{align*}
                \|{^{\pm}}{I}{^{1-\alpha}} Du_j - {^{\pm}}{I}{^{1-\alpha}} \mathcal{D}u\|_{L^{p}(\Omega)} \leq C \|Du_j - \mathcal{D}u \|_{L^{p}(\Omega)}, 
            \end{align*}
            which converges to zero by the choice of $\{u_j\}_{j=1}^{\infty}$. Setting $j\to \infty$ in the above 
            equation yields that ${^{\pm}}{\mathcal{D}}{^{\alpha}} u= - {^{\pm}}{I}{^{1-\alpha}} \mathcal{D}u$. 
            Thus, 
            \begin{align*}
                \|{^{\pm}}{\mathcal{D}}{^{\alpha}} u\|_{L^p(\Omega)} = \| {^{\pm}}{I}{^{1-\alpha}} \mathcal{D}u \|_{L^{p}(\Omega)} \leq C \|\mathcal{D}u\|_{L^p(\Omega)}<\infty.
            \end{align*}
            The proof is complete. 
        \end{proof}
        
        \begin{remark}
        	From the above proof we can see that $W^{1,p}(\R) \subset {^{\pm}}{W}{^{\alpha,p}}(\R)$
        	for all $0 <\alpha<1$ and $1 \leq p < \infty$ with the replacing of $\{u_j\}_{j=1}^{\infty} \subset C^{\infty}_{0}(\R)$.
        \end{remark}
        
        To see that the need for zero boundary traces is a necessary condition, we consider  the function 
        $u \equiv 1$. With $\Omega = (a,b)$ is a finite domain, it is easy to check that $u$ is weakly differentiable with 
        the weak derivative coinciding with the Riemann-Liouville derivative, that is,  ${^{-}}{\mathcal{D}}{^{\alpha}} 1 = \Gamma(1-\alpha)^{-1} (x-a)^{-\alpha}$ 
        and a similar formula holds for the right weak derivative. It is easy to show that $\|{^{\pm}}{\mathcal{D}}{^{\alpha}} 1 \|_{L^{p}((a,b))} < \infty$ if and only if $\alpha p <1$. 
        Therefore, the inclusion $W^{1,p}((a,b)) \subset {^{\pm}}{W}{^{\alpha,p}}((a,b))$ may not 
        hold in general. However, the next theorem shows that the inclusion does hold in general
        provided that $\alpha p <1$. 
          
        \begin{theorem}
           Let $\Omega =(a,b)$, $0 <\alpha<1$ and $1\leq p<\infty$. Suppose that $\alpha p <1$. Then $W^{1,p}(\Omega) \subset {^{\pm}}{W}{^{\alpha,p}}(\Omega)$. Hence, $W^{1,p}(\Omega) \subset   {W}{^{\alpha,p}}(\Omega)$ when $\alpha p <1$.
        \end{theorem}
        
        \begin{proof}
            We only give a proof for $W^{1,p}(\Omega) \subset {^{-}}{W}{^{\alpha,p}}(\Omega)$ because 
             the inclusion $W^{1,p}(\Omega)$ $\subset {^{+}}{W}{^{\alpha,p}}(\Omega)$ can be proved similarly.
            
            Let $u \in W^{1,p}((a,b))$. By the density/approximation theorem, there exists a sequence $\{u_j\}_{j=1}^{\infty} \subset C^{\infty}((a,b))\cap C([a,b])$ so that 
            $u_j \rightarrow u$ in $W^{1,p}((a,b))\cap C([a,b])$. 
            Then for any $\varphi \in C^{\infty}_{0}((a,b))$, using the integration by parts formula and the 
            relationship between the Riemann-Liouville and Caputo derivatives, we get 
            \begin{align*}
              \int_{a}^{b} u_j(x) {_{x}}{D}{^{\alpha}_{b}} \varphi(x)\,dx 
                &=\int_{a}^{b} {_{a}}{D}{^{\alpha}_{x}}u_j(x) \varphi (x)\,dx \\
                &= \int_{a}^{b} \left(\dfrac{u_{j}(a)  }{\Gamma(1- \alpha) (x-a)^{\alpha}} + {_{a}}{I}{^{1-\alpha}_{x}} Du_j(x)\right) \varphi(x)\,dx.
            \end{align*}
            Taking the limit $j\to \infty$ on both sides yields 
            \begin{align*}
                \int_{a}^{b} u(x) {_{x}}{D}{^{\alpha}_{b}} \varphi(x)\, dx 
                = \int_{a}^{b} \biggl( \dfrac{u(a) }{\Gamma(1-\alpha) (x-a)^{\alpha}} + {_{a}}{I}{^{1-\alpha}_{x}} \mathcal{D}u(x)\biggr) \varphi(x)\,dx.
            \end{align*}
            Hence, ${^{-}}{\mathcal{D}}{^{\alpha}} u$ almost everywhere in $(a,b)$ and is given by
            \begin{equation}\label{weak_Caputo}
            {^{-}}{\mathcal{D}}{^{\alpha}} u = \dfrac{u(a) }{\Gamma(1-\alpha) (x-a)^{\alpha}} 
            + {_{a}}{I}{^{1-\alpha}_{x}} \mathcal{D}u(x).
            \end{equation}
            It remains to verify that ${^{-}}{\mathcal{D}}{^{\alpha}} u \in L^{p}((a,b))$, which 
            can be easily done for $\alpha p<1$ using the formula above for the weak derivative and 
            the mapping properties of the fractional integral operators. The proof is complete.
        \end{proof}
    
    \begin{remark}
    \eqref{weak_Caputo} suggests the following definitions of the weak fractional Caputo derivatives for any $u\in W^{1,1}(\Omega)$:
    \begin{align}\label{weak_Caputo_derivativeL}
    {^{-}_{C}}{\mathcal{D}}{^{\alpha}} u(x) &: = {_{a}}{I}{^{1-\alpha}_{x}} \mathcal{D}u(x) \qquad\mbox{a.e. in }  \Omega, \\
    {^{+}_{C}}{\mathcal{D}}{^{\alpha}} u(x) &: =  {_{x}}{I}{^{1-\alpha}_{b}} \mathcal{D}u(x) \qquad\mbox{a.e. in }  \Omega, \label{weak_Caputo_derivativeR}
    \end{align} 
    and then we have almost everywhere in $\Omega$
    \begin{align}\label{C1}
     {^{-}_{C}}{\mathcal{D}}{^{\alpha}} u (x)&: =  {^{-}}{\mathcal{D}}{^{\alpha}} u(x)  - \dfrac{u(a)  }{\Gamma(1-\alpha) (x-a)^{\alpha}},\\
      {^{+}_{C}}{\mathcal{D}}{^{\alpha}} u(x) &: = {^{+}}{\mathcal{D}}{^{\alpha}} u(x)   - \dfrac{u(b)  }{\Gamma(1-\alpha)(b-x)^{\alpha} }.  \label{C2}
    \end{align} 
    \end{remark}
    
    We conclude this section with an integration by parts formula for functions in one-sided Sobolev spaces. The need to wait until now for such a formula will be evident in the assumptions.
    
    \begin{proposition}
        Let $\Omega \subset \R$, $\alpha >0$, $1 \leq p < \infty$. Suppose that $u \in {^{\pm}}{W}{^{\alpha,p}}(\Omega)$, $v \in W^{1,q}_{0}(\Omega)$, and $w \in W^{1,q}(\Omega)$. Then there holds 
        \begin{align}\label{Sobolev0IBP}
            \int_{\Omega} v{^{\pm}}{\mathcal{D}}{^{\alpha}} u\, dx 
             = (-1)^{[\alpha]}\int_{\Omega} u {^{\mp}}{\mathcal{D}}{^{\alpha}}v \, dx. 
        \end{align}
        Moreover,  if $\alpha q < 1$, there holds
        \begin{align}\label{SobolevIBP}
            \int_{\Omega} w{^{\pm}}{\mathcal{D}}{^{\alpha}} u \,dx = (-1)^{[\alpha]} \int_{\Omega} u {^{\mp}}{\mathcal{D}}{^{\alpha}} w\, dx.
        \end{align}
    \end{proposition}
    
    \begin{proof}
        We only give a proof for \eqref{Sobolev0IBP} when $0 < \alpha <1$. The other cases and (\ref{SobolevIBP}) can be showed similarly. 
        Choose $\{u_j\}_{j=1}^{\infty} \subset {}{C}^{\infty}(\Omega)$
        	and $\{v_k\}_{k=1}^{\infty} \subset C^{\infty}_0(\Omega)$ 
        such that $u_j\to u$ in ${^{\pm}}{W}{^{\alpha,p}}(\Omega)$ and  
          $v_k \rightarrow v$ in $W^{1,q}(\Omega)$. By Theorem \ref{TraceZeroConsistency} we have 
        $v \in {^{\mp}}{W}{^{\alpha,q}}(\Omega)$. It follows that 
        \begin{align*}
            \int_{\Omega} u {^{\mp}}{\mathcal{D}}{^{\alpha}} v\, dx = \lim_{j,k \rightarrow \infty} \int_{\Omega} u_j {^{\mp}}{D}{^{\alpha}} v_k\, dx 
            = \lim_{j,k \rightarrow \infty} \int_{\Omega} v_k {^{\pm}}{D}{^{\alpha}} u_j \, dx
            = \int_{\Omega} v {^{\pm}}{\mathcal{D}}{^{\alpha}} u_j\ dx.
        \end{align*}
        This completes the proof.
    \end{proof}

    \subsubsection{\bf The Case $p=1$ and $\Omega = \R$}\label{sec-4.6.2}
       First, by doing a change of variables we get for any $\varphi \in C^{\infty}_{0}(\R)$ 
        \begin{align*}
                {^{-}}{D}{^{\alpha}} \varphi(x) &= \dfrac{1}{\Gamma(1- \alpha)} \dfrac{d}{dx} \int_{-\infty}^{x} \dfrac{\varphi(y)}{(x-y)^{\alpha}} \,dy 
                = \dfrac{1}{\Gamma(1 - \alpha)} \dfrac{d}{dx} \int_{0}^{\infty} t^{-\alpha} \varphi(x-t)\,dt\\ 
                &= \dfrac{1}{\Gamma(1- \alpha)} \int_{0}^{\infty} t^{-\alpha} \varphi'(x-t)\,dt  
                = \dfrac{\alpha}{\Gamma(1- \alpha)} \int_{-\infty}^{x} \dfrac{\varphi(x) - \varphi(t)}{(x-t)^{1+\alpha}}\,dt.
        \end{align*}
        Similarly, 
        \begin{align*}
            {^{+}}{D}{^{\alpha}}\varphi(x) &= \dfrac{\alpha}{\Gamma(1 - \alpha)} 
            \int_{x}^{\infty} \dfrac{ \varphi(t) - \varphi(x) }{(t-x)^{1 + \alpha}}\,dt.
        \end{align*}
        These equivalent formulas will be used in the proof of the next theorem. 

        \begin{theorem}
            Let $0<\alpha<1$. Then $\widetilde{W}^{\alpha,1}(\R) \subseteq {^{\pm}}{W}{^{\alpha,1}}(\R)$.
        \end{theorem}
        
        \begin{proof}
            Let $u \in \widetilde{W}^{\alpha,1}(\R)$. Recall that $C^{\infty}_{0}(\R)$ is dense in $\widetilde{W}^{\alpha,1}(\R)$. Then 
            there exists a sequence $\{u_j\}_{j=1}^{\infty} \subset C^{\infty}_{0}(\R)$ such that $u_j \rightarrow u$ 
            in $\widetilde{W}^{\alpha,1}(\R)$ as $j \rightarrow \infty$. We only give a proof of the inclusion 
            for the left fractional Sobolev space because the proof for the other case follows similarly.  
            
            Using the above equivalent formula for left derivatives, we get 
            \begin{align*}
                \left\| {^{-}}{\mathcal{D}}{^{\alpha}}u_j \right\|_{L^{1}(\R)} &= \left\|{^{-}}{D}{^{\alpha}} u_j \right\|_{L^{1}(\R)} 
                = C_\alpha \int_{\R} \left|\int_{-\infty}^{x} \dfrac{u_j (x) - u_j(y)}{(x-y)^{1+\alpha}}\,dy \right|dx \\ 
                &\leq C_\alpha \int_{\R} \int_{-\infty}^{x} \dfrac{|u_j(x) - u_j(y)|}{|x-y|^{1+\alpha}}\,dydx \\
                &\leq C_\alpha \int_{\R} \int_{\R} \dfrac{|u_j(x) - u_j(y)|}{|x-y|^{1+\alpha}}\,dydx  \\
                &= C_\alpha \left[u_j \right]_{\widetilde{W}^{\alpha,1}(\R)}.
            \end{align*}
            By the property of $\{u_j\}_{j=1}^{\infty}$, we conclude that $\left[ u_j \right]_{\widetilde{W}^{\alpha,1}(\R)} \rightarrow [u]_{\widetilde{W}^{\alpha,1}(\R)} < \infty$. 

            Let $\eps > 0$, for sufficiently large $m,n \in \N$, we have
            \begin{align*}
                \left\|{^{\pm}}{\mathcal{D}}{^{\alpha}} u_m - {^{\pm}}{\mathcal{D}}{^{\alpha}} u_n \right\|_{L^{1}(\R)} = \left\|{^{\pm}}{\mathcal{D}}{^{\alpha}} \left[u_m - u_n \right]\right\|_{L^{1}(\R)}  
                \leq \left[u_m - u_n \right]_{\widetilde{W}^{\alpha,1}(\R)}  
                <\eps.
            \end{align*}
            Hence, $\{u_j\}_{j=1}^{\infty}$ is a Cauchy sequence in ${^{\pm}}{W}{^{\alpha,1}}(\R)$. Thus, 
            there exists $v \in {^{\pm}}{W}{^{\alpha,1}}(\R)$ such that $u_j \rightarrow v$ 
            in ${^{\pm}}{W}{^{\alpha,1}}(\R)$. By the property of $\{u_j\}_{j=1}^{\infty}$, there holds 
            $u_j \rightarrow u$ in $L^{1}(\R)$. On the other hand, the convergence in ${^{\pm}}{W}{^{\alpha,1}}(\R)$ implies that $u_j \rightarrow v$ in $L^{1}(\R)$.
            Thus, $u=v$ almost everywhere in $\R$ and yielding that $u \in {^{\pm}}{W}{^{\alpha,1}}(\R)$. 
        \end{proof}

    \subsubsection{\bf The Case $p=2$ and $\Omega = \R$}\label{sec-4.6.3}
        This section extends the above equivalence result of two fractional Sobolev spaces to the case when $p=2$. 
        As we will see, $p=2$ is special in the sense that it is the only case in which the equivalence of the space $\widehat{H}^{\alpha}(\R)$ defined by the Fourier transform (and its inverse) and the space $\widetilde{H}^{\alpha}(\R)$ holds. Recall that $\widehat{W}^{\alpha,p}(\R) \neq \widetilde{W}^{\alpha,p}(\R)$ for $p\neq2$ (cf. \cite{Adams,Brezis}).

        \begin{lemma}\label{equivalence_seminorms}
           Let $0 < \alpha <1$ and $\varphi \in C^{\infty}_{0}(\R)$. Then $\|{^{\mathcal{F}}}{D}{^{\alpha}} \varphi \|_{L^{2}(\R)} \cong [ \varphi ]_{\widetilde{H}^{\alpha}(\R)}$.
        \end{lemma}
        
        \begin{proof}
            Let $\hat{\varphi}=\mathcal{F}(\varphi)$. It follows from Plancherel theorem and (\ref{SeminormRelation}) that
            \begin{align*}
                \|{^{\mathcal{F}}}{D}{^{\alpha}} \varphi \|_{L^{2}(\R)}^{2} &= \|\mathcal{F}^{-1} [ (i\xi)^{\alpha} \hat{\varphi} ] \|_{L^{2}(\R)}^{2} 
                = \|(i\xi)^{\alpha} \hat{\varphi} \|_{L^{2}(\R)}^{2}  \\ 
                & = \int_{\R} |i\xi|^{2\alpha} |\hat{\varphi}(\xi)|^{2} \,d\xi 
                = \int_{\R} |\xi|^{2\alpha} |\hat{\varphi}(\xi)|^{2}\,d\xi  \\
                &\cong
                [u]_{\widetilde{H}^{\alpha}(\R)}.
            \end{align*}
            Taking the square root of each side, we obtain the desired result. 
        \end{proof}

        \begin{theorem}
            Let $0 < \alpha <1$. Then ${^{\pm}}{{H}}{^{s}} (\R) = \widetilde{H}^{s}(\R)$. 
        \end{theorem}
        
        \begin{proof}
            Step 1: Suppose $u\in{^{\pm}}{{H}}{^{\alpha}}(\R)$. Since $C^{\infty}_{0}(\R)$ is dense in ${^{\pm}}{{H}}{^{\alpha}}(\R)$, then there exists a sequence $\left\{u_j \right\}_{j=1}^{\infty} \subset C^{\infty}_{0}(\R)$ such that $ u_j \rightarrow u$ in ${^{\pm}}{{H}}{^{\alpha}}(\R)$. 
            Then by Lemma \ref{equivalence_seminorms} we get 
            \begin{align*}
                \left\|u _j \right\|_{\widetilde{H}^{\alpha}(\R)}^{2} &= \left\| u _j \right\|_{L^{2}(\R)}^{2} + \left[ u_j \right]_{\widetilde{H}^{\alpha}(\R)}^{2} 
                \leq  \left\|u_j \right\|_{L^{2}(\R)}^{2} + C \| {^{\mathcal{F}}}{D}{^{\alpha}} u_j \|_{L^{2}(\R)}^{2}   \\ 
                &=    \left\|u _ j\right\|_{L^{2}(\R)}^{2} + C\left\|{^{\pm}}{\mathcal{D}}{^{\alpha}} u_j \right\|_{L^{2}(\R)}^{2}  
                \leq  C \left\|u_j \right\|_{{^{\pm}}{{H}}{^{\alpha}}(\R)}^{2}.
            \end{align*}
            Consequently,
            \begin{align*}
                \|u_m - u_n\|_{\widetilde{H}^{\alpha}(\R)} \leq C \|u_m - u_n \|_{{^{\pm}}{H}{^{\alpha}}(\R)}\to 0 \quad\mbox{as } m,n\to \infty.
            \end{align*}
             Thus, $\{u_j\}_{j=1}^{\infty}$ is a Cauchy sequence in $\widetilde{H}^{\alpha}(\R)$. Since $\widetilde{H}^{\alpha}(\R)$ is a Banach space, there exists $v \in \widetilde{H}^{\alpha}(\R)$ so that $u_j \rightarrow v$ in $\widetilde{H}^{\alpha}(\R)$; in particular, $u_j \rightarrow v$ in $L^{2}(\R)$. By assumption, $u_j \rightarrow u$ in $L^{2}(\R)$. Therefore, $v = u$ a.e. in $\R$ and $u \in \widetilde{H}^{\alpha}(\R)$.

            Step 2: Let $u \in \widetilde{H}^{\alpha}(\R)$. By the approximation theorem, there exists
            a sequence $\{u_j\}_{j=1}^{\infty} \subset C^{\infty}_{0}(\R)$ 
            such that $u_j \rightarrow u$ in $\widetilde{H}^{\alpha}(\R)$. Then by Lemma \ref{equivalence_seminorms} we get 
            \begin{align*}
                \|u_j\|_{{^{\pm}}{H}{^{\alpha}}(\R)}^{2} &= \|u_j\|_{L^{2}(\R)}^{2} + \|{^{\pm}}{\mathcal{D}}{^{\alpha}} u_j \|_{L^{2}(\R)}^{2} \\ 
                &= \|u_j\|_{L^{2}(\R)}^{2} + \|{^{\mathcal{F}}}{D}{^{\alpha}}u_j \|_{L^{2}(\R)}^{2} \\ 
                &\leq \|u_j\|_{L^{2}(\R)}^{2} + C [ u_j]_{\widetilde{H}^{\alpha} (\R)}^{2}.
            \end{align*}
            It implies that 
            \begin{align*}
                \|u_m - u_n \|_{{^{\pm}}{H}{^{\alpha}}(\R)} \leq C \| u_m - u_n \|_{\widetilde{H}^{\alpha}(\R)}  \to 0 \quad\mbox{as } m,n\to \infty.
            \end{align*}
            Hence $\{u_j\}_{j=1}^{\infty}$ is a Cauchy sequence in ${^{\pm}}{H}{^{\alpha}}(\R)$. Since ${^{\pm}}{H}{^{\alpha}}(\R)$ is a Banach space, there exists $v \in {^{\pm}}{H}{^{\alpha}}(\R)$ so that $u_j \rightarrow v$ in ${^{\pm}}{H}{^{\alpha}}(\R)$; in particular, $u_j \rightarrow v$ in $L^{2}(\R)$. By assumption $u_j \rightarrow u$ in $L^{2}(\R)$. Therefore, $v=u$ a.e. and $u \in {^{\pm}}{H}{^{\alpha}}(\R)$. 
        \end{proof}

  \begin{remark}
        (a) The above result immediately infers that the equivalences \newline 
         ${^{\pm}}{H}{^{\alpha}}(\R) = \widetilde{H}^{\alpha}(\R) = \widehat{H}^{\alpha}(\R)$.
        
      (b) We note that  ${^{-}}{{H}}{^{s}} (\R) = {^{+}}{{H}}{^{s}} (\R)$, 
          however, this does not means 
        that the left and right weak derivatives of the same function are the same 
        or equivalent but rather two spaces contain the same set of functions. 
    
      (c) We conjecture that ${^{\pm}}{W}{^{\alpha,p}}(\R) \neq \widehat{W}^{\alpha,p}(\R)$, but 
      ${^{\pm}}{W}{^{\alpha,p}}(\R) = {\tW}^{\alpha,p}(\R)$ for $p\neq2$ and $0<\alpha<1$.
      
      (d) It can easily be shown that the equality ${^{\pm}}{W}{^{\alpha,p}}(\Omega) = {\tW}^{\alpha,p}(\Omega)$ cannot hold in general. It was proved that when $\alpha p > 1$, ${^{\pm}}{\mathcal{D}}{^{\alpha}} C \notin {^{\pm}}{W}{^{\alpha,p}}(\Omega)$. However, constant functions always belong to ${\tW}^{\alpha,p}(\Omega)$. In general, ${\tW}^{\alpha,p}(\Omega) \not\subset {^{\pm}}{W}{^{\alpha,p}}(\Omega)$. For the same reason, ${\tW}^{\alpha,p}(\Omega) \not\subset  {W}^{\alpha,p}(\Omega)$ when $\alpha p >1$. This simple example shows that the fractional derivative definition is fundamentally different from the (double) integral term resembling a difference quotient in the seminorm of ${\tW}^{\alpha,p}(\Omega)$. If an equivalence exists on the finite domain, it is our conjecture that for $\alpha p <1$, the spaces $ {W}^{\alpha , p}(\Omega)$ and ${\tW}^{\alpha,p}(\Omega)$ are the two spaces that should be comparable. 
      
  \end{remark}


\section{Conclusion}\label{sec-5}
   In this paper we introduced three families of new fractional Sobolev
   spaces based on the newly developed weak fractional derivative 
   notion in \cite{Feng_Sutton, Feng_Sutton1}, they were defined in the exact same manner 
   as done for the integer order Sobolev spaces. Many important theorems and properties,
    such as density theorem, extension theorems, one-sided trace theorem, various embedding 
    theorems and Sobolev inequalities, integration by parts formulas and dual space 
    characterizations in those Sobolev spaces were established. 
    Moreover, a few relationships, including equivalences and differences, with existing 
   fractional Sobolev spaces were also established.
   
   It is expected (and our hope, too) that these newly developed theories of weak fractional differential calculus and fractional order Sobolev spaces will lay down a solid theoretical foundation for systematically and rigorously developing a fractional calculus of variations theory and a fractional PDE theory as well as their numerical solutions. Moreover, we hope this work will stimulate more research on and applications of fractional calculus and fractional differential equations, including the extensions to higher dimension, in the near future.


\end{document}